\documentclass{article}

\usepackage{a4}
\usepackage{amssymb}
\usepackage{QED}
\usepackage{url}

\newtheorem{theorem}{Theorem}

\newtheorem{corollary}[theorem]{Corollary}

\newtheorem{definition}[theorem]{Definition}

\newtheorem{lemma}[theorem]{Lemma}

\newtheorem{remark}[theorem]{Remark}

\def\be{\beta}
\def\o{\omega}

\def\l{\lambda}
\def\o{\omega}
\def\r{\restriction}
\def\u{\uparrow}

\def\G{\Gamma}
\def\D{\Delta}
\def\B{\nabla}
\def\f{\rightarrow}
\def\fx{\leadsto}

\def\v{\vdash}
\def\<{\langle}
\def\>{\rangle}
\def\F{\displaystyle\frac}

\def\rb{\rhd_{\beta}}

\title{A completeness result for a realisability semantics for an intersection type system}

\author{Fairouz Kamareddine\thanks{School of Mathematical and Computer Sciences,
        Heriot-Watt Univ., Riccarton,
        Edinburgh EH14 4AS, Scotland, fairouz@macs.hw.ac.uk} $\;\;\;$
        and$\;\;\;$
Karim Nour\thanks{Universit\'e de Savoie, Campus Scientifique,
73378 Le Bourget du Lac, France, nour@univ-savoie.fr}}

\begin{document}
\maketitle

\begin{abstract}
In this paper we consider a type system with a universal type $\o$
where any term (whether open or closed, $\beta$-normalising or not)
has type $\o$.  We provide this type system with a realisability
semantics where an atomic type is interpreted as the set of
$\lambda$-terms saturated by a certain relation.  The variation of the
saturation relation gives a number of interpretations to each type.
We show the soundness and completeness of our semantics and that for
different notions of saturation (based on weak head reduction and
normal $\beta$-reduction) we obtain the same interpretation for types.
Since the presence of $\o$ prevents typability and realisability from
coinciding and creates extra difficulties in characterizing the
interpretation of a type, we define a class ${\mathbb U}^+$ of the
so-called positive types (where $\o$ can only occur at specific
positions).  We show that if a term inhabits a positive type, then
this term is  $\beta$-normalisable and reduces to a closed term.
In other words, positive types can be used to represent abstract data
types.  The completeness theorem for ${\mathbb U}^+$ becomes
interesting indeed since it establishes a perfect equivalence between
typable terms and terms that inhabit a type.  In other words,
typability and realisability coincide on ${\mathbb U}^+$.  We give a
number of examples to explain the intuition behind the definition of
${\mathbb U}^+$ and to show that this class cannot be extended while
keeping its desired properties.

 \end{abstract}

\section{Introduction}
\label{intrp}

The ground work for intersection types and related notions was developed in the seventies
 \cite{coppo78,coppo80,salle} and have since proved to be a valuable tool in
 the theoretical studies and applications of the lambda calculus.
 Intersection types incorporate type polymorphism in a finitary way
 (where the usage of types is listed rather than quantified over).
Since the late seventies, numerous intersection type systems have been developed or used 
for a multitude of purposes (the list is huge; for a very brief list we simply refer the reader to the recent articles \cite{alessi,carwel} and the references there, for a longer list we refer the reader to the 
 bibliography of intersection types and related systems available (while that URL address is active) at
\url{http://www.macs.hw.ac.uk/~jbw/itrs/bibliography.html}).
In this paper, we are interested in the
 interpretation of an intersection type.  We study this interpretation
 in the context of the so-called realisability semantics.

The idea of realisability semantics is to associate to each type a set
of terms which realise this type. Under this semantics, an atomic type is interpreted as the set of
$\lambda$-terms saturated by a certain relation.  Then, arrow and
intersection types receive their intuitive interpretation of
functional space and set intersection. For example, a term which realises
the type ${\mathbb N}\rightarrow {\mathbb N}$ is a function from
${\mathbb N}$ to ${\mathbb N}$.  Realisability semantics has been a
powerful method for establishing the strong normalisation of type
systems \`a la Tait and Girard.  The realisability of a type system
enables one to also show the soundness of the system in the sense that
the interpretation of a type contains all the terms that have this
type.  Soundness has been an important method for characterising the
algorithmic behaviour of typed terms through their types as has been
illuminative in the work of Krivine.

It is also interesting to find the class of types for which the
converse of soundness holds.  I.e., to find the types $A$ for which
the realisability interpretation contains exactly (in a certain sense)
the terms typable by $A$. This property is called completeness and
has not yet been studied for every type system.

In addition to the questions of soundness and completeness for a
realisability semantics, one is interested in the additional three
questions:
\begin{enumerate}
\item 
Can different interpretations of a type given by different saturation relations be compared?
\item
For a particular saturation relation, what are the types uniquely realised by the $\lambda$-terms which are typable by these types?
\item
Is there a class of types for which typability and realisability coincide?
\end{enumerate}
In this paper we establish the soundness and completeness as well as give answers to questions 1, 2 and 3 for a strict non linear intersection type system with a universal type.  We show that for different notions of saturation (based on weak head
reduction and normal $\beta$-reduction) we obtain the same
interpretation for types answering question 1 partially. Questions 2
and 3 are affected by the presence of $\o$ which prevents typability
and realisability from coinciding and creates extra difficulties in
characterizing the interpretation of a type.  We define a class
${\mathbb U}^+$ of the so-called positive types (where $\o$ can only
occur at specific positions).  We show that if a term inhabits a
positive type, then this term is  $\beta$-normalisable and
reduces to a closed term.  In other words, positive types can be used
to represent abstract data types.  This result answers question 2
and depends on the full power of soundness.  The completeness theorem
for ${\mathbb U}^+$ becomes interesting indeed since it establishes a
perfect equivalence between typable terms and terms that inhabit a
type.  In other words, typability and realisability coincide on
${\mathbb U}^+$ answering question 3.  We give a number of examples to
explain the intuition behind the definition of ${\mathbb U}^+$ and to show
that this class cannot be extended while keeping its desired properties.

 Hindley~\cite{Hin2, Hin3,
Hin4} was the first to study the completeness of a simple type system
and he showed that all the types of that system have the completeness
property.  Then, he generalised his completeness proof for an
intersection type system~\cite{Hin1}.  
Using his completeness theorem for the realisability
semantics based on the sets of $\lambda$-terms saturated by
$\beta\eta$-equivalence, Hindley has shown that simple types have
property 2 above. However, his completeness theorem for intersection
types does not allow him to establish property 2 for the intersection
type system. Moreover, Hindley's completeness theorems were
established with the sets of $\lambda$-terms saturated by
$\beta\eta$-equivalence, and hence they don't permit a comparison
between the different possible interpretations. In our method,
saturation is not by $\beta\eta$-equivalence.  Rather, it is by the
weaker requirement of weak head normal forms. Hence, all of Hindley's
saturated models are also saturated in our framework and moreover,
there are saturated models based on weak head normal form which cannot
be models in Hindley's framework.

\cite{Lab1} has established
completeness for a class of types in Girard's system F (also
independently discovered by Reynolds as the second order typed
$\lambda$-calculus) known as the strictly positive types.
\cite{FaNo1, FaNo2} generalised the result of~\cite{Lab1} for the
larger class which includes all the positive types and also for second
order functional arithmetic. \cite{Coq} established
recently by a different method using Kripke models, the completeness
for the simply typed $\lambda$-calculus. Finally \cite{NoSa}
introduced a realizability semantics for the simply typed
$\lambda\mu$-calculus and proved a completeness result.

The paper is structured as follows: In section~\ref{typesystem}, we
introduce the intersection type system that will be studied in this
paper.  In section~\ref{subredsec} we study both the subject reduction and subject expansion 
properties for $\beta$.  In
section~\ref{realsem} we establish the soundness and completeness of
the realisability semantics based on two notions of saturated sets
(one using weak head reduction and the other using $\beta$-reduction).
In section~\ref{meantypes} we show that the meaning of a type does not
depend on the chosen notion of saturation (based on either weak head
reduction or $\beta$-reduction).  We also define a subset of types
which we show to satisfy the (weak) normalisation property and for
which typability and realisability coincide.

\section{The typing system}
\label{typesystem}

A number of intersection type systems have been given in the
literature (for a very brief list see \cite{alessi,carwel} and the references there; for a longer list (and while that URL address is active) see \url{http://www.macs.hw.ac.uk/~jbw/itrs/bibliography.html}).  In
this paper we introduce an interesection type system due to J.B.\
Wells and inspired by his work with S\'ebastien Carlier on expansion \cite{carwel}.  
We follow
\cite{carwel} and write the type judgements $\G \v M: U$ as $M: \<\G
\v U\>$.  There are many reasons why this latter notation is to be prefered over the former (see \cite{carwel}). In particular, 
this typing notation allowed J.B. Wells in \cite{principal} to give a very simple yet general definition of principal typings.

Before presenting the type system, we give a number of its characteristics:
\begin{itemize}
\item
The type system is {\em relevant}: this means that the type environments contain all and only the necessary assumptions 
as is shown in lemma~\ref{structyping}.\ref{structypingone}. 
\item
The type system is {\em strict}
and {\em non-linear}.   Following the terminology of \cite{bakel} (who advocated the use of of linear systems of intersection types only with strict intersection types), types are strict if $\o$ and $\sqcap$ do not occur
immediately to the right of arrows.   Our type system is non-linear since $\sqcap$ is
idempotant. We guarantee strictness by using two sets of types
${\mathbb T}$ and ${\mathbb U}$ such that ${\mathbb T} \subset
{\mathbb U}$ and ${\mathbb T}$ is only formed by either basic types or
using the arrow constructor (without permitting $\o$ and $\sqcap$ to
occur immediately to the right of arrows).  This means that one does  not need to state laws relating 
$A \rightarrow (B_1 \sqcap B_2)$ to $(A \f B_1) \sqcap (A \f B_2)$, yet one can still establish a number of type inclusion properties as is shown in lemma~\ref{goodeg}.
\end{itemize}

\begin{definition}
\begin{enumerate}
\item Let ${\cal V}$ be a  denumerably infinite set of
variables. The set of terms  ${\cal M}$, of the $\l$-calculus is defined
as usual by the following grammar:

$${\cal M} ::= \; {\cal V} \; \mid \; (\l {\cal V}.{\cal M}) \; \mid
\; ({\cal M} {\cal M})$$
We let $x, y, z, etc.$ range over ${\cal V}$ and $M, N, P, Q, M_1, M_2,
\dots$ range over ${\cal M}$.  
We assume the the usual definition of subterms and the usual
convention for parenthesis and omit these when no confusion arises.
In particular, we write $M\;N_1...N_n$ instead of $(...(M \; N_1)\;N_2...N_{n-1})\;N_n$.

We take terms modulo
$\alpha$-conversion and use the
Barendregt convention (BC) where the names of bound
variables differ from the free ones. When two terms $M$ and $N$ are equal
(modulo $\alpha$), we write $M=N$.
We write $FV(M)$ for the set of the free variables of term $M$.

\item We define as usual the substitution $M[x:= N]$ of the term $N$
for all free occurrences of $x$ in the term $M$ and similarly,
$M[(x_i:=N_i)_{1}^n]$, the simultaneous substitution of $N_i$ for all
free occurrences of $x_i$ in $M$ for $1\leq i \leq n$.

\item We assume the usual definition of compatibility.
\begin{itemize}
\item
The weak head reduction $\rhd_f$ on ${\cal M}$  is defined by: $M
\rhd_f N$ if $M = (\l x. P)Q \;Q_1...Q_n$ and $N = P[x:=Q]\;
Q_1...Q_n$ where $n \geq 0$.
\item
The reduction relation $\rhd_\be$ on ${\cal M}$  is defined as the
least compatible relation closed under  the rule: $(\l x. M)N
\rhd_\be  M[x:=N]$.
\item For $r \in \{f,\be\}$, $\rhd_r^*$ denotes the
reflexive transitive closure of $\rhd_r$.
\item $\simeq_{\beta}$ denotes the equivalence relation induced by
$\rhd_{\beta}^*$. 
\end{itemize}
\end{enumerate}
\end{definition}

The next theorem is standard and is needed for the rest of the paper.
\begin{theorem}\label{conf}
\begin{enumerate}
\item\label{confone}
Let $r \in \{f,\be\}$.  If $M \rhd_r^* N$, then $FV(N) \subseteq
FV(M)$.
\item\label{conftwo} If $M \rhd^*_f N$, then, for all $P \in {\cal M}$, $M P \rhd^*_f N
P$.
\item\label{confthree}
If $M \rhd^*_\be M_1$ and $M \rhd^*_\be M_2$, then there is
$M'$ such that $M_1 \rhd^*_\be M'$ and $M_2 \rhd^*_\be M'$.
\item
\label{confthree'}
$M_1 \simeq_{\beta} M_2$ iff there is a
term $M$ such that $M_1 \rhd^*_\be M$ and $M_2 \rhd^*_\be M$.
\item\label{conftwo'} 
Let $n \geq 1$ and assume 
$x_i \not \in  FV(M)$ for every $1 \leq i \leq n$.  If 
$M x_1...x_n \rb^* x_j \; N_1...N_m$ for some $1 \leq j \leq
n$ and $m \geq 0$,  then for some $k \geq j$ and $s \leq m$, $M \rb^*
\l x_1....\l x_k. x_j \; M_1...M_s$ where $s + n  = k + m$, $M_i
\simeq_{\beta} N_i$ for every $1 \leq i \leq s$ and $N_{s+i}
\simeq_{\beta} x_{k+i}$ for every $1 \leq i \leq n-k$.
\item\label{conftwo''}  If $M\, x$ is weakly $\beta$-normalising  and $x \not \in FV(M)$, then $M$ is also weakly $\beta$-normalising.
\end{enumerate}
\end{theorem}

\begin{proof}
See~\cite{Bar1} for more detail.  Here, we sketch the proofs.  
\ref{confone} (resp.\ \ref{conftwo}) is by induction on $M \rhd_r^* N$ (resp.\ $M \rhd_f^* N$).  \ref{confthree} is the Church-Rosser. \ref{confthree'} if) is by definition of $\simeq_{\beta}$ whereas only if) is by induction on $M_1 \simeq_{\beta} M_2$  using \ref{confthree}.  
\begin{itemize}
\item[\ref{conftwo'}.] is as follows: Since $M x_1 ... x_n \rhd^*_\be x_j N_1
... N_m$, then by page 23 of \cite{Kri1}, $M x_1 ... x_n$ is solvable and hence, $M$ is also
solvable and its head reduction terminates. Therefore, $M
\rhd^*_\be \lambda x_1 ... \lambda x_k. z M_1 ... M_s$ for $s, k \geq 0$.
Since $ x_j \; N_1...N_m  \simeq_{\beta}  ( \lambda x_k. z M_1 ... M_s)x_1...x_n $ then $k \leq n$, 
$ x_j \; N_1...N_m  \simeq_{\beta}  z M_1 ... M_s x_{k+1}...x_n$.  Hence, $z = x_j$, $s \leq m$, $j \leq k$ (since $x_j \not \in FV(M)$), $m = s + (n - (k +
1)) + 1 = s + n - k$, $M_i
\simeq_{\beta} N_i$ for every $1 \leq i \leq s$ and $N_{s+i}
\simeq_{\beta} x_{k+i}$ for every $1 \leq i \leq n-k$.
\item[\ref{conftwo''}.] is by cases:
\begin{itemize}
\item
If $M\, x \rhd^*_\be M'\, x$ where $M'\, x$ is in $\be$-normal form and $M  \rhd^*_\be M'$ then $M'$ is in $\be$-normal form and  $M$ is $\beta$-normalising.
\item
If $M\, x \rhd^*_\be (\l y.N)\, x \rhd_\be N[y:=x] \rhd^*_\be P$ where $P$ is in $\be$-normal form and $M  \rhd^*_\be \l y.N$ then by \ref{confone}, $x \not \in FV(N)$ and so, 
$ M  \rhd^*_\be \l y.N = \l x.N[y:=x] \rhd^*_\be \l x.P$.
 Since $\l x.P$ is  in $\be$-normal form, $M$ is $\beta$-normalising.
\end{itemize}
\end{itemize}
\end{proof}

\begin{definition}
\begin{enumerate}
\item Let ${\cal A}$ be a denumerably infinite set of atomic types.
The types are defined by the following grammars:
$${\mathbb T} ::= \; {\cal A} \; \mid \; {\mathbb U} \f {\mathbb T}$$
$${\mathbb U} ::= \; \o \; \mid \; {\mathbb U} \sqcap {\mathbb U}
\; \mid \; {\mathbb T}$$

We let $a, b, c, a_1, a_2, \dots$ range over ${\cal A}$,
 $T, T_1, T_2, T', \dots$ range over ${\mathbb T}$ and $U$,
$V$, $W$, $U_1$, $V_1$, $U', \dots$ range over ${\mathbb U}$.

We quotient types by taking $\sqcap$ to be commutative (i.e. $U_1
\sqcap U_2 = U_2 \sqcap U_1$), associative (i.e. $U_1 \sqcap (U_2
\sqcap U_3) = (U_1 \sqcap U_2) \sqcap U_3$), idempotent (i.e. $U
\sqcap U = U$) and to have $\o$ as neutral (i.e. $\o \sqcap U =
U$).\\ We denote $U_n\sqcap U_{n+1} \dots \sqcap U_{m}$ by
$\sqcap_{i=n}^{m}U_i$ (when $n \leq m$).

\item A type environment is a set $\{x_i:U_i$ / $1 \leq i \leq n, n \geq 0,
\mbox{ and }   \forall 1 \leq i \leq n, \; x_i \in {\cal V}, \;
U_i \in {\mathbb U} \mbox{ and } \forall 1 \leq i, j \leq n,
\mbox{ if } i \not = j \mbox { then } x_i \neq x_j\}$. We denote
such environment (call it $\G$) by $x_1:U_1, \dots ,x_n:U_n$ or
simply by $(x_i:U_i)_n$ and define $dom(\G) = \{x_i$ / $1 \leq i
\leq n\}$.  We use $\G, \D, \G_1, \dots$ to range over
environments and write $()$ for the empty environment.

If $M$ is a
term and $FV(M) = \{x_1,...,x_n\}$, we denote $env^M_\o  = (x_i :
\o)_n$.

If $\G = (x_i:U_i)_n$, $x \not \in dom(\G)$ and $U \in
{\mathbb U}$, we denote $\G, x:U$ the type environment $x_1:U_1,
\dots ,x_n:U_n,x:U$.

Let $\G_1 = (x_i:U_i)_n,(y_j:V_j)_m$ and
$\G_2 = (x_i:U'_i)_n,(z_k:W_k)_l$. We denote $\G_1 \sqcap \G_2$
the type environment $(x_i:U_i \sqcap
U'_i)_n,(y_j:V_j)_m,(z_k:W_k)_l$. Note that $dom(\G_1 \sqcap \G_2)
= dom(\G_1) \cup dom(\G_2)$ and that $\sqcap$ is commutative,
associative and idempotent on environments.

\item The typing rules are the following:

\begin{center}
$\F{}{x : \<x : T \v T\>}\;\;\;ax$\\[0,5cm]
\end{center}

\begin{center}
$\F{}{M : \<env^M_\o \v \o \>}\;\;\;\o$\\[0,5cm]
\end{center}

\begin{center}
$\F{M : \<\G , x :U \v T\>}{\l x. M : \<\G \v U \f T\>}\;\;\;\f_i$\\[0,5cm]
\end{center}

\begin{center}
$\F{M : \<\G \v T\>\;\;\;x \not \in dom(\G)}{\l x. M : \<\G \v \o \f T\>}\;\;\;\f'_i$\\[0,5cm]
\end{center}

\begin{center}
$\F{M_1 : \<\G_1 \v U \f T\> \;\;\; \hspace{0.2in} M_2 : \<\G_2 \v
U\>}{M_1 \; M_2 : \<\G_1 \sqcap \G_2 \v T\>}\;\;\;\f_e$\\[0,5cm]
\end{center}

\begin{center}
$\F{M: \<\G \v U_1\> \;\;\;\hspace{0.2in} M : \<\G \v U_2\>}
{M : \<\G \v U_1 \sqcap U_2\>}\;\;\;\sqcap_i$\\[0,5cm]
\end{center}

\begin{center}
$\F{M : \<\G\v U\> \;\;\;\hspace{0.2in} \<\G\v U\> \sqsubseteq \<\G'\v U'\>}
{M : \<\G'\v U'\>}\;\;\;\sqsubseteq$\\[1cm]
\end{center}
In the last clause, the binary relation $\sqsubseteq$ is defined
by the following rules:

\begin{center}
$\F{}{\Phi \sqsubseteq \Phi}\;\;\;ref$ \\[0.5cm]

$\F{\Phi_1 \sqsubseteq \Phi_2 \;\;\; \; \;\;\; \Phi_2 \sqsubseteq \Phi_3}{\Phi_1 \sqsubseteq \Phi_3}\;\;\;tr$ \\[0.5cm]

$\F{}{U_1 \sqcap U_2 \sqsubseteq U_1}\;\;\;\sqcap_e$\\[0.5cm]

$\F{U_1 \sqsubseteq V_1 \;\;\;  \;\;\; U_2 \sqsubseteq V_2}{U_1 \sqcap U_2 \sqsubseteq V_1 \sqcap V_2}\;\;\;\sqcap$ \\[0.5cm]

$\F{U_2 \sqsubseteq U_1 \;\;\;  \;\;\; T_1 \sqsubseteq T_2}{U_1 \f T_1 \sqsubseteq U_2 \f T_2}\;\;\;\f$ \\[0.5cm]

$\F{U_1 \sqsubseteq U_2 \;\;\;  \;\;\; x \not \in  dom(\G)}{\G, x : U_1 \sqsubseteq  \G, x : U_2}\;\;\;\sqsubseteq_c$ \\[0.5cm]

$\F{U_1 \sqsubseteq U_2 \;\;\; \;\;\; \G_2 \sqsubseteq
\G_1}{\<\G_1 \v  U_1\> \sqsubseteq
\<\G_2 \v  U_2\>}\;\;\;\sqsubseteq_{\<\>}$ \\[0.5cm]
\end{center}

Throughout, we use $\Phi, \Phi', \Phi_1, \dots$ to denote $U \in
{\mathbb U}$, or environments $\G$ or typings $\<\G\v U\>$. Note
that when $\Phi \sqsubseteq \Phi'$, then $\Phi$ and $\Phi'$ belong
to the same set (either ${\mathbb U}$ or environments or typings).
\end{enumerate}
\end{definition}

The next lemma gives the shape of a type in ${\mathbb U}$.

\begin{lemma}\label{omega}
\begin{enumerate}
\item\label{omegaone}
If $U \in {\mathbb U}$, then $U = \omega$ or $U =
\sqcap_{i=1}^{n}T_i$ where $n \geq 1$ and $\forall~1 \leq i \leq
n$, $T_i \in {\mathbb T}$.
\item \label{omegatwo} $U \sqsubseteq \o$.
\item\label{omegathree}  If $\omega \sqsubseteq U$, then $U = \omega$.
\end{enumerate}
\end{lemma}

\begin{proof}
\begin{itemize}
\item[\ref{omegaone}.]  By induction on $U \in  {\mathbb U}$.
\item[\ref{omegatwo}.] By rule $\sqcap_e$, $U = \o \sqcap U \sqsubseteq \o$. 
\item[\ref{omegathree}.] By
induction on the derivation  $\omega \sqsubseteq U$.
\end{itemize}
\end{proof}

The next lemma studies the relation $\sqsubseteq$ on ${\mathbb U}$.
\begin{lemma}\label{goodeg}
Let $V \neq \o$.
\begin{enumerate}
\item  \label{goodegone}
If $U \sqsubseteq V$, then $U = \sqcap_{j=1}^k T_j$, $V =
\sqcap_{i=1}^p T'_i$ where $p, k \geq 1 $, $\forall 1 \leq j\leq k$, $1\leq i \leq p$, $T_j, T'_i \in {\mathbb T}$,
 and $\forall~1 \leq i \leq
p$, $\exists 1 \leq j \leq k$ such that $T_j \sqsubseteq T'_i$.
\item\label{goodegone'}
If $U \sqsubseteq V' \sqcap a$, then $U = U'\sqcap a$ and $U'  \sqsubseteq V'$.
\item\label{goodegtwo'}
Let $p, k \geq 1$. If $\sqcap_{j=1}^k (U_j \f T_j) \sqsubseteq
\sqcap_{i=1}^p (U'_i \f T'_i)$, then  $\forall 1 \leq i \leq p$,
$\exists 1 \leq j \leq k$ such that $U'_i  \sqsubseteq U_j$ and
$T_j \sqsubseteq T'_i$.
\item\label{goodegtwo''}
If $U \f T \sqsubseteq V$, then $V = \sqcap_{i=1}^p (U_i \f T_i)$
where $p \geq 1$ and $\forall 1 \leq i \leq p$, $U_i \sqsubseteq
U$ and $T \sqsubseteq T_i$.
\item\label{goodegthree}
If $\sqcap_{j=1}^k (U_j \f T_j) \sqsubseteq V$ where $k \geq 1$,
then $V = \sqcap_{i=1}^p (U'_i \f T'_i)$ where $p \geq 1$ and
$\forall 1 \leq i \leq p$, $\exists 1 \leq j \leq k$ $U'_i
\sqsubseteq  U_j$ and $T_j \sqsubseteq T'_i$.

\end{enumerate}
\end{lemma}

\begin{proof}
\begin{itemize}
\item[\ref{goodegone}.] By induction on the derivation $U \sqsubseteq V$ using lemma~\ref{omega}.\ref{omegaone}.
\item[\ref{goodegone'}.]  By induction on $U \sqsubseteq V' \sqcap a$.
\item[\ref{goodegtwo'}.]
By induction on $\sqcap_{j=1}^k (U_j \f T_j) \sqsubseteq \sqcap_{i=1}^p (U'_i \f T'_i) $.
We only do the {\it tr} case.\\
If $\F{\sqcap_{j=1}^k (U_j \f T_j) \sqsubseteq V \;\;\;\; V
\sqsubseteq \sqcap_{i=1}^p (U'_i \f T'_i)}{\sqcap_{j=1}^k (U_j \f
T_j) \sqsubseteq \sqcap_{i=1}^p (U'_i \f T'_i)}$, then, by
~\ref{goodegone}, $V = \sqcap_{l=1}^q T''_l$ where $q \geq 1$ and
$\forall 1 \leq l \leq q$, $\exists 1 \leq j \leq k$, such that
$U_j \f T_j \sqsubseteq T''_l$.  If $T''_l = a$, then, by
\ref{goodegone'},  $U_j \f T_j =  U'\sqcap a$. Absurd. Hence,
$\forall 1 \leq l \leq q$, $T''_l = V_l \f T'''_l$ and $V =
\sqcap_{l=1}^q (V_l \f T'''_l)$. 
Let $1 \leq i \leq p$.  By IH, 
$\exists 1 \leq l \leq q$, $U'_i \sqsubseteq V_l$ and $T'''_l
\sqsubseteq T'_i$. Also, by IH, 
$\exists 1 \leq j \leq k$, $V_l \sqsubseteq U_j$ and $T_j
\sqsubseteq T'''_l$.  
Hence, $\forall 1 \leq i \leq p$, $\exists 1
\leq j \leq k$, such that
$U'_i \sqsubseteq U_j$ and $T_j \sqsubseteq T'_i$.
\item[\ref{goodegtwo''}.] By~\ref{goodegone}, $V = \sqcap_{i=1}^p T'_i$
where $p \geq 1$ and $\forall 1 \leq i \leq p$, $U \f T
\sqsubseteq T'_i$.  If $T'_i = a$, then, by \ref{goodegone'},  $U
\f T =  U'\sqcap a$. Absurd. Hence, $T'_i = U_i \f T_i$.  Hence,
by~\ref{goodegtwo'}, $\forall 1 \leq i \leq p$,
$U_i \sqsubseteq  U$ and $T \sqsubseteq T_i$.
\item[\ref{goodegthree}.] Since $V \not = \o$, then, by
lemma~\ref{omega}.\ref{omegaone}, $V = \sqcap_{i=1}^p T'_i$ where
$p \geq 1$ and $\forall 1 \leq i \leq p$, $T'_i \in {\mathbb T}$.
Let $1 \leq i \leq p$.  By ~\ref{goodegone}, $\exists 1 \leq j_i
\leq k$ such that $U_{j_i} \f T_{j_i} \sqsubseteq T'_i$.
By~\ref{goodegtwo''}, and since $T'_i \in {\mathbb T}$, $T'_i =
U'_i \f T''_i$ where $U'_i \sqsubseteq U_{j_i}$ and $T_{j_i}
\sqsubseteq T''_i$.  Hence, $V = \sqcap_{i=1}^p (U'_i \f T''_i)$
where $p \geq 1$ and $\forall 1 \leq i \leq p$, $\exists 1 \leq
j_i \leq k$ $U'_i \sqsubseteq  U_{j_i}$ and $T_{j_i} \sqsubseteq
T''_i$.

\end{itemize}
\end{proof}

The next lemma studies the relation $\sqsubseteq$ on environments and typings.

\begin{lemma}\label{Phisub}
\begin{enumerate}
\item \label{Phisubone'} If $\G \sqsubseteq \G'$, then $dom(\G) = dom(\G')$.
\item\label{Phisubone}
If $\G \sqsubseteq \G'$, $U \sqsubseteq U'$ and $x \not \in dom(\G)$,
then $\G, x:U \sqsubseteq \G', x:U'$.
\item \label{Phisubtwo}
$\G \sqsubseteq \G'$ iff $\G = (x_i : U_i)_n$, $\G' = (x_i : U'_i)_n$ and for
every $1 \leq i \leq n$, $U_i \sqsubseteq U'_i$.
\item \label{Phisubtwo'}
If $dom(\G) = FV(M)$, then
$\G \sqsubseteq env_\o^M$
\item \label{Phisubtwo''}
If $env_\o^M \sqsubseteq \G$, then $\G = env_\o^M$.
\item \label{Phisubthree} $\<\G \v  U\> \sqsubseteq \<\G' \v  U'\>$ iff
$\G' \sqsubseteq \G$ and  $U \sqsubseteq U'$.
\item \label{Phisublast}
If $\G \sqsubseteq \G'$ and $\Delta  \sqsubseteq \Delta'$, then
$\G\sqcap \Delta \sqsubseteq \G'\sqcap \Delta'$.
\end{enumerate}
\end{lemma}

\begin{proof}
\begin{itemize}
\item[\ref{Phisubone'}.] By induction on the derivation $\G \sqsubseteq \G'$. 
\item[\ref{Phisubone}.] First show, by induction on the derivation $\G
\sqsubseteq \G'$ (using \ref{Phisubone'}), that if  $\G
\sqsubseteq \G'$, $V \in {\mathbb U}$ and $y \not \in dom(\G)$
then $\G, y:V \sqsubseteq \G',  y:V$. Then use tr.
\item[\ref{Phisubtwo}.]
Only if) By  \ref{Phisubone'}, $\G = (x_i : U_i)_n$ and $\G' = (x_i : U_i)_n$. The proof is
by induction on the derivation $(x_i : U_i)_n \sqsubseteq (x_i :
U'_i)_n$. If)  By induction on $n$ using  \ref{Phisubone}.
\item[\ref{Phisubtwo'}.]
Let $FV(M) =\{x_1, \dots, x_n\}$ and  $\G = (x_i :U_i)_n$. By definition,
$env_\o^M = (x_i,\o)_n$.  Hence, by lemma~\ref{omega}.\ref{omegatwo} and \ref{Phisubtwo},  $\G \sqsubseteq env_\o^M$.
\item[\ref{Phisubtwo''}.]
Let $FV(M) =\{x_1, \dots, x_n\}$.
By definition,
$env_\o^M = (x_i,\o)_n$.  By  \ref{Phisubtwo},   $\G = (x_i :U_i)_n$ and $\forall 1 \leq i \leq n$, $\omega \sqsubseteq U_i$. Hence by lemma~\ref{omega}.\ref{omegathree}, $\forall 1 \leq i \leq n$, $\omega = U_i$.
\item[\ref{Phisubthree}.]
Only if) By induction on the derivation $\<\G \v U\> \sqsubseteq \<\G' \v U'\>$. If) By $\sqsubseteq_{\<\>}$.
\item[\ref{Phisublast}.] 
 This is a corollary of~\ref{Phisubtwo}.
\end{itemize}
\end{proof}

The next lemma shows that we do not allow weakening in our type system.
\begin{lemma}\label{structyping}
\begin{enumerate}
\item \label{structypingone} If $M: \<\G \v U\>$, then $dom(\G) = FV(M)$.
\item \label{structypingtwo} For every $\G$ and $M$ such that $dom(\G) = FV(M)$, we have  $M: \<\G \v
\o\>$.
\end{enumerate}
\end{lemma}

\begin{proof}
\begin{itemize}
\item[\ref{structypingone}.] By induction on the derivation $M: \<\G \v U\>$. 
\item[\ref{structypingtwo}.] By
$\o$, $M: \<env_\o^M \v \o\>$. By lemma~\ref{Phisub}.\ref{Phisubtwo'},
$\G \sqsubseteq env_\o^M$.  Hence, by  $\sqsubseteq$ and $\sqsubseteq_{\<\>}$, $M: \<\G \v \o\>$.
\end{itemize}
\end{proof}

Finally, it may come as a surprise that the rule {\it ax} uses types in ${\mathbb T}$ instead of ${\mathbb U}$ and that in  the rule $\sqcap$ we take the same environment.  The lemma below shows that this is not restrictive.
\begin{lemma}\label{newrules}
\begin{enumerate}
\item\label{newrulesone}
The rule
$\F{M: \<\G_1 \v  U_1\> \;\;\;\hspace{0.2in} M : \<\G_2 \v  U_2\>}
{M : \<\G_1 \sqcap \G_2 \v  U_1 \sqcap U_2\>}\;\;\;\sqcap'_i$  is derivable.
\item\label{newrulestwo}
The rule $\F{}{x : \<(x : U) \v U\>}\;\;\;ax'$ is derivable.
\end{enumerate}
\end{lemma}

\begin{proof}
\begin{itemize}
\item[\ref{newrulesone}.] Let $M: \<\G_1 \v  U_1\>$ and $M : \<\G_2 \v  U_2\>$. By
lemma~\ref{structyping}, $dom(\G_1) = dom(\G_2) = FV(M)$.  Let $\G_1 = (x_i: V_i)_n$ and $\G_2 = (x_i: V'_i)_n$.  Hence, $\G_1 \sqcap \G_2 = (x_i: V_i \sqcap V'_i)_n$.
By $ V_i \sqcap V'_i \sqsubseteq V_i$ and $ V_i \sqcap V'_i \sqsubseteq V'_i$ for all $1 \leq i \leq n$.  Hence,
by lemma~\ref{Phisub}.\ref{Phisubtwo},
$\G_1 \sqcap \G_2 \sqsubseteq\G_1$ and
$\G_1 \sqcap \G_2 \sqsubseteq\G_2$, and, by rules $\sqsubseteq$
and $\sqsubseteq_{\<\>}$, $M :\<\G_1 \sqcap \G_2, U_1\>$ and $M
:\<\G_1 \sqcap \G_2, U_2\>$. Finally, by rule $\sqcap_i$,
$M :\<\G_1 \sqcap \G_2, U_1\sqcap U_2\>$.
\item[\ref{newrulestwo}.]
By lemma~\ref{omega}.\ref{omegaone}:
\begin{itemize}
\item
Either  $U = \o$, then, by rule $\o$, we have $x:\<(x : \o)\v
\omega\>$.
\item
Or $U = \sqcap_{i=1}^k T_i$ where $\forall 1 \leq i
\leq k$, $T_i \in {\mathbb T}$, then, by rule $ax$, $x : \<(x :
T_i) \v T_i\>$ and, by $k-1$ applications of rule $\sqcap'_i$, $x
: \<(x : U) \v U\>$.
\end{itemize}
\end{itemize}
\end{proof}

\section{Subject reduction and expansion properties}
\label{subredsec}
In this section we establish the subject reduction and subject expansion properties for $\beta$.
\subsection{Subject reduction for $\beta$}
We start with a form of the generation lemma.

\begin{lemma}[Generation]\label{newgen}
\begin{enumerate}
\item\label{newgenone} If $x : \<\G \v  U \>$, then $\G = (x : V)$ and $V \sqsubseteq U$.
\item\label{newgenfour} If $M \; x  : \<\G , x : U \v  V\>$ and $x \not \in
FV(M)$,
then $V = \o$ or $V = \sqcap_{i=1}^k T_i$ where $k \geq 1$ and
$\forall 1 \leq i \leq k$, $M  : \<\G \v  U \f T_i\>$.
\item\label{newgentwo} If $\l x.M : \<\G \v  U\>$ and $x \in FV(M)$, then $U = \o$ or $U =
\sqcap_{i=1}^k (V_i \f T_i)$ where $k \geq 1$ and $\forall 1 \leq
i \leq k$, $M  : \<\G , x : V_i \v T_i\>$.
\item\label{newgenthree} If $\l x.M : \<\G \v  U\>$ and $x \not \in FV(M)$, then $U = \o$ or $U =
\sqcap_{i=1}^k (V_i \f T_i)$ where $k \geq 1$ and $\forall 1 \leq
i \leq k$, $M  : \<\G \v T_i\>$.
\end{enumerate}
\end{lemma}

\begin{proof}
\ref{newgenone}.\ By induction on  the derivation $x : \<\G \v  U \>$. We have
four cases:
\begin{itemize}
\item If $\F{}{x:\<(x : T) \v T\>}$, nothing to prove.
\item If $\F{}{x:\<(x : \o) \v \o\>}$, nothing to prove.
\item Let $\F{x:\<\G \v U_1\>\;\;\;x:\<\G \v U_2\>}{x:\<\G \v U_1 \sqcap
U_2\>}$. By IH, $\G = (x : V)$, $V \sqsubseteq U_1$ and $V
\sqsubseteq U_2$, then, by rule $\sqcap$, $ V \sqsubseteq U_1
\sqcap U_2$.
\item Let $\F{x:\<\G' \v U'\>\;\;\;\<\G' \v U'\> \sqsubseteq \<\G \v U\> }{x:\< \G \v
U\>}$. By lemma~\ref{Phisub}.\ref{Phisubthree}, $\G \sqsubseteq \G'$ and $U'
\sqsubseteq U$ and, by IH, $\G' = (x : V')$ and $V' \sqsubseteq
U'$. Then, by lemma~\ref{Phisub}.\ref{Phisubtwo}, $\G = (x : V)$, $V
\sqsubseteq V'$ and, by rule $tr$, $V \sqsubseteq U$.
\end{itemize}
\ref{newgenfour}.\ By induction on the derivation $M \, x : \<\G ,x : U \v V\>$.
We have four cases:
\begin{itemize}
\item
If $\F{}{M \; x : \<env_\o^{M \; x} \v \omega\>}$, nothing to
prove.

\item
Let $\F{M: \<\G \v  U \f T\> \;\;\; x : \<(x : V)  \v  U\>}{M \, x
: \<\G , x : V \v T\>}$ (where, by 1.\, $V \sqsubseteq U$).

Since $U \f T \sqsubseteq V \f T$, we have $M: \<\G \v V \f T\>$.

\item
Let $\F{M \; x: \<\G, x:U  \v  U_1\> \; \; M \; x : \<\G, x:U \v
U_2\>} {M \; x : \<\G, x:U  \v  U_1 \sqcap U_2\>}$. By IH, we have four
cases:
\begin{itemize}
\item If $U_1 = U_2 = \omega$, then $U_1 \sqcap U_2 = \omega$.
\item If $U_1 = \omega$, $U_2 = \sqcap_{i=1}^k T_i$, $k \geq 1$ and
$\forall 1 \leq i \leq k$, $M  : \<\G \v  U \f T_i\>$, then $U_1
\sqcap U_2 = U_2$ ($\o$ is a neutral element).
\item If $U_2 = \omega$, $U_1 = \sqcap_{i=1}^k T_i$, $k \geq 1$ and
$\forall 1 \leq i \leq k$, $M  : \<\G \v  U \f T_i\>$, then $U_1
\sqcap U_2 = U_1$ ($\o$ is a neutral element).
\item If $U_1 =  \sqcap_{i=1}^k T_i$ and $U_2 = \sqcap_{i=1}^l
T_{k+i}$ (hence $U_1 \sqcap U_2 = \sqcap_{i=1}^{k + l} T_i$) ,
where $k,l \geq 1$ and $\forall 1 \leq i \leq k+l$, $M : \<\G \v U
\f T_i\>$.
\end{itemize}
\item
Let $\F{M \; x : \<\G', x:U' \v V'\> \;\;\; \<\G', x:U' \v V'\>
\sqsubseteq \<\G, x:U \v V\>}{M \; x : \<\G, x :U \v V\>}$ (by
lemma~\ref{Phisub}).

By lemma~\ref{Phisub}, $\G \sqsubseteq \G'$, $U\sqsubseteq
U'$ and $V' \sqsubseteq V$. By IH, we have two cases:
\begin{itemize}
\item If $V' = \omega$, then, by
lemma~\ref{omega}.\ref{omegathree}, $V = \omega$.
\item If $V' = \sqcap_{i=1}^k T'_i$, where $k\geq 1$ and $\forall 1 \leq
i \leq k$, $M : \<\G \v  U \f T'_i\>$. By lemma~\ref{goodeg}.\ref{goodegone},
$V = \omega$ (nothing to prove) or $V = \sqcap_{i=1}^p T_i$ where
$p \geq 1$ and $\forall 1 \leq i \leq p$, $\exists 1 \leq j_i \leq k$ such that
$T'_{j_i} \sqsubseteq
T_i$. Since, by lemma~\ref{Phisub}.\ref{Phisubthree}, $\<\G' \v  U' \f T'_{j_i}
\>\sqsubseteq \<\G \v U \f T_i \>$ for any $1 \leq i \leq p$, then
$\forall 1 \leq i \leq p$, $M : \<\G \v  U \f T_i\>$.
\end{itemize}
\end{itemize}
\ref{newgentwo}.\ By induction on  the derivation $\l x.M : \<\G \v  U\>$. We
have four cases:
\begin{itemize}
\item
If $\F{}{\l x.M : \<env_\o^{\l x.M} \v  \o\>}$, nothing to prove.
\item
If $\F{M : \<\G, x : U \v  T\>}{\l x. M : \<\G \v U \f T\>}$,
nothing to prove.
\item
Let $\F{\l x.M: \<\G \v  U_1\> \; \; \l x.M : \<\G \v  U_2\>} {\l
x.M : \<\G \v  U_1 \sqcap U_2\>}$. By IH, we have four cases:
\begin{itemize}
\item If $U_1 = U_2 = \o$, then $U_1 \sqcap U_2 = \o$.
\item If $U_1 = \omega$, $U_2 = \sqcap_{i=1}^k (V_i \f T_i)$
where $k \geq 1$ and $\forall 1 \leq i \leq k$, \\ $M  : \<\G_2,x :
V_i \v T_i\>$, then $U_1 \sqcap U_2 = U_2$ ($\o$ is a neutral element).
\item If $U_2 = \omega$, $U_1 = \sqcap_{i=1}^k (V_i \f T_i)$
where $k \geq 1$ and $\forall 1 \leq i \leq k$, \\
$M  : \<\G_1 ,x :
V_i \v T_i\>$, then $U_1 \sqcap U_2 = U_1$ ($\o$ is a neutral element).
\item If $U_1 = \sqcap_{i= 1}^k (V_i
\f T_i)$,  $U_2 = \sqcap_{i= k+1}^{k+l} (V_{i} \f T_{i})$ (hence
$U_1\sqcap U_2 = \sqcap_{i= 1}^{k+l} (V_i \f T_i)$) where $k,l
\geq 1$, $\forall 1 \leq i \leq k+l$, $M : \<\G , x : V_i \v T_i\>$, we are done.
\end{itemize}
\item
Let $\F{\l x.M : \<\G\v U\> \;\;\;  \<\G\v U\> \sqsubseteq \<\G'\v
U'\>}{\l x.M : \<\G'\v U'\>}$.  By lemma~\ref{Phisub}.\ref{Phisubthree}, $\G'
\sqsubseteq \G$ and $U \sqsubseteq U'$. By IH, we have two cases:
\begin{itemize}
\item If $U = \omega$, then, by
lemma~\ref{omega}.\ref{omegathree}, $U' = \omega$.
\item Assume $U = \sqcap_{i=1}^k (V_i \f T_i)$, where $k\geq
1$ and $M : \<\G , x : V_i \v  T_i\>$ for all $1 \leq i \leq k$.
By lemma~\ref{omega}.\ref{omegaone}:
\begin{itemize}
\item
Either $U' = \omega$, and hence nothing to prove.
\item
Or, by lemma~\ref{goodeg}.\ref{goodegthree}, $U' = \sqcap_{i=1}^p
(V'_i \f T'_i)$, where $p\geq 1$ and $\forall 1 \leq i \leq p$,
$\exists 1 \leq j_i \leq k$ such that $V'_i \sqsubseteq V_{j_i}$
and $T_{j_i} \sqsubseteq T'_i$. Let $1 \leq i \leq p$. Since, by
lemma~\ref{Phisub}.\ref{Phisubthree}, $\<\G ,x : V_{j_i} \v
T_{j_i} \>\sqsubseteq \<\G' , x : V'_i \v T'_i \>$, then $M :
\<\G', x : V'_i \v T'_i\>$.
\end{itemize}
\end{itemize}
\end{itemize}
\ref{newgenthree}.\ Same proof as that of \ref{newgentwo}.
\end{proof}
Now, we establish the substitution lemma.
\begin{lemma}[Substitution]
\label{substlem} If $M:\<\G,x:U\v  V\>$ and $N :\<\D \v  U\>$,\\
then $M[x:=N]:\<\G\sqcap \D \v V\>$.
\end{lemma}
\begin{proof}
By induction on the derivation $M:\<\G,x:U\v  V\>$.
\begin{itemize}
\item
If $\F{}{x : \<(x:T) \v  T\>}$ and $N :\<\D\v T\>$, then $N =
x[x:=N]:\<\D \v  T\>$.
\item
If $\F{}{M : \<(x_i:\o)_n, x:\o \v \o\>}$ where $FV(M)= \{x_1, \dots, x_n, x\}$  and if $N :\<\D\v \o\>$, then since $FV(M[x:=N]) = \{x_1,\dots, x_n\} \cup FV(N)$, we have by $\o$, $M[x:=N]:\<(x_i:\o)_n\sqcap env^N_\o\v \o\>$.
By lemmas~\ref{Phisub}.\ref{Phisubtwo'} and~\ref{structyping}, $\D \sqsubseteq   env^N_\o$ and by lemma~\ref{Phisub}.\ref{Phisublast}, 
$(x_i:\o)_n\sqcap \D \sqsubseteq (x_i:\o)_n\sqcap env^N_\o$.  Hence, by $\sqsubseteq_{\<\>}$,
$M[x:=N]:\<(x_i:\o)_n\sqcap \D \v \o\>$.
\item
Let $\F{M : \<\G, x:U, y:U' \v  T\>}{\l y. M : \<\G, x:U \v U' \f
T\>}$. By IH, $M[x:=N] : \<\G\sqcap \D, y:U' \v  T\>$. By rule
$\f_i$, $(\l y.M)[x:=N] = \l y.M[x:=N] : \<\G\sqcap \D \v U'\f
T\>$.
\item
Let $\F{M : \<\G, x:U \v  T\> \;\;\;y \not \in
dom(\G)\cup\{x\}}{\l y. M : \<\G, x:U \v \omega \f T\>}$. By IH,
$M[x:=N] : \<\G\sqcap \D \v T\>$. By rule $\f'_i$, $(\l y.M)[x:=N]
= \l y.M[x:=N] : \<\G\sqcap \D \v \omega \f T\>$.
\item
Let $\F{M_1 : \<\G_1, x:U_1 \v  V \f T\> \;\;\; M_2 : \<\G_2,
x:U_2 \v  V\>} {M_1 \; M_2 : \<\G_1 \sqcap \G_2, x:U_1\sqcap U_2
\v T\>}$ where $x \in FV(M_1)\cap FV(M_2)$ and $N :\<\D\v
U_1\sqcap U_2\>$. By rules $\sqcap_e$ and $\sqsubseteq$, $N
:\<\D\v U_1\>$ and $N :\<\D\v  U_2\>$. Now use IH and
rule $\f_e$.\\
The cases $x \in FV(M_1)\setminus FV(M_2)$ or $x \in
FV(M_2)\setminus FV(M_1)$ are easy.
\item
If $\F{M: \<\G, x:U \v  U_1\> \;\; M : \<\G, x:U \v  U_2\>} {M :
\<\G, x:U \v  U_1 \sqcap U_2\>}$ use IH and $\sqcap_i$.
\item
Let $\F{M : \<\G', x:U' \v V'\> \;\;\; \<\G', x:U' \v  V'\>
\sqsubseteq \<\G, x:U \v V\>}{M : \<\G, x:U\v V\>}$ (by
lemma~\ref{Phisub}).

By  lemma~\ref{Phisub}, $dom(\G) = dom(\G')$, $\G \sqsubseteq
\G'$, $U \sqsubseteq U'$ and $V' \sqsubseteq V$. Hence by $\sqsubseteq$,  $N :\<\D\v
U'\>$ and, by IH, $M[x:=N] : \<\G'\sqcap \D\v V'\>$. It is easy to
show $\G \sqcap \D \subseteq \G'\sqcap \D$. Hence, $\<\G'\sqcap
\D\v V'\>\sqsubseteq \<\G\sqcap \D\v  V\>$ and  by $\sqsubseteq$, $M[x:=N] : \<\G\sqcap
\D\v V\>$.
\end{itemize}
\end{proof}

Since our system does not allow weakening, we need the next definition (and the related lemma below it) since when a term is reduced, it may lose some of its free variables and hence will need to be typed in a smaller environment.
\begin{definition}
If $\G$ is a type environment and ${\cal U} \subseteq dom(\G)$,
then  we write $\G\r_{\cal U}$ for the restriction of $\G$ on the
variables of ${\cal U}$.  If ${\cal U} = FV(M)$ for a term $M$,
we write $\G \r_M$ instead of $\G \r_{FV(M)}$.
\end{definition}

\begin{lemma}
\label{restlem}
\begin{enumerate}
\item\label{restlemone}
If $FV(N) \subseteq FV(M)$, then $ env^M_\o \r_N = env^N_\o$.
\item\label{restlemtwo}
If $FV(M) \subseteq dom(\G_1)$ and  $FV(N) \subseteq dom(\G_2)$, then \\
$(\G_1 \sqcap\G_2 )\r_{MN} \sqsubseteq (\G_1 \r_M)\sqcap\G_2$.
\end{enumerate}
\end{lemma}
\begin{proof}
\ref{restlemone}.\ Easy.  \ref{restlemtwo}.\ First, note that
$dom((\G_1 \sqcap\G_2 )\r_{MN}) = FV(MN) = FV(M) \cup FV(N) = dom(\G_1
\r_M) \cup dom(\G_2) = dom( (\G_1 \r_M)\sqcap\G_2)$. Now, we show by cases that
if $x:U_1 \in (\G_1 \sqcap\G_2 )\r_{MN}$ and $x:U_2 \in (\G_1 \r_M)\sqcap\G_2$
then  $U_1 \sqsubseteq U_2$:
\begin{itemize}
\item If $x \in FV(M)\cap FV(N)$ then  $x:U'_1 \in \G_1$,  $x:U''_1 \in \G_2$ and $U_1 = U'_1\sqcap U''_1 = U_2$.
\item If $x \in FV(M)\setminus FV(N)$ then $x \not \in dom(\G_2)$,  $x:U_1 \in \G_1$ and $U_1 = U_2$.
\item If $x \in FV(N)\setminus FV(M)$ then 
\begin{itemize}
\item If $x \in dom(\G_1)$ then  $x:U'_1 \in \G_1$,  $x:U_2 \in \G_2$ and  $U_1 = U'_1 \sqcap U_2\sqsubseteq U_2$.

\item If $x \not \in dom(\G_1)$ then $x:U_2 \in \G_2$ and $U_1 = U_2$.
\end{itemize}
\end{itemize}
\end{proof}

Now we give  the basic block in the subject reduction for $\beta$.
\begin{theorem}\label{betaetatheo}
If $M : \<\G \v  U\>$ and $M \rhd_{\beta} N$, then $N : \<\G\r_N
\v  U\>$.
\end{theorem}

\begin{proof}
By induction on the derivation $M : \<\G \v  U\>$. Rule $\o$
follows  by theorem~\ref{conf}.\ref{confone} and
lemma~\ref{restlem}.\ref{restlemone}. Rules
$\f_i$, $\f'_i$, $\sqcap_i$ and $\sqsubseteq$ are by IH. We do $\f_e$\\
Let $\F{M_1 : \<\G_1 \v  U \f T\> \;\;\; Q : \<\G_2 \v  U\>}
{M_1 \; Q : \<\G_1 \sqcap \G_2 \v T\>}$.
\begin{itemize}
\item
If $M = M_1Q \rhd_\beta  PQ= N$ where $M_1 \rhd_\beta P$ then by IH, 
$P  : \<\G_1 \r_{P} \v  U \f T\>$.  By  $\f_e$, $P \; Q : \<(\G_1\r_{P}) \sqcap \G_2 \v T\>$.
By  lemma~\ref{restlem}.\ref{restlemtwo}, $(\G_1 \sqcap\G_2 )\r_{PQ} \sqsubseteq (\G_1 \r_{P})\sqcap\G_2$.
Finally, by $\sqsubseteq_{\<\>}$,  $P \; Q : \< (\G_1 \sqcap\G_2 )\r_{PQ} \v T\>$.
\item
The case  $M = M_1Q \rhd_\beta  M_1P=N$ where $Q \rhd_\beta P$ is similar to the above.
\item
 Assume $M_1 = \l x.P$ and $M_1\;M_2 = (\l
x.P)M_2 \rhd_\beta P[x := M_2] = N$. Since $\l x.P: \<\G_1 \v U \f
T\>$, we have two cases:
\begin{itemize}
\item If $x \in FV(P)$, then, by lemma~\ref{newgen}.\ref{newgentwo},
$P: \<\G_1, x:U \v T\>$. By lemma~\ref{substlem}, $P[x := M_2]:
\<\G_1\sqcap \G_2 \v T\>$.  Moreover,  $FV(M_1M_2) = FV(N) = dom(\G_1 \sqcap \G_2)$.
 Hence $(\G_1\sqcap \G_2)\r_N = \G_1\sqcap \G_2$ and
$N:\<(\G_1\sqcap \G_2)\r_N \v T\>$.

\item If $x \not \in FV(P)$, then, by lemma~\ref{newgen}.\ref{newgenthree},
$P: \<\G_1 \v T\>$.  Moreover, by lemma~\ref{structyping}.\ref{structypingone},
$FV(P) = FV(M_1) = dom(\G_1)$. Hence, $(\G_1 \sqcap \G_2)\r_P = \G_1\r_P = \G_1$
 and $P[x := M_2] = P : \<(\G_1 \sqcap \G_2)\r_P
\v T\>$.
\end{itemize}
\end{itemize}
\end{proof}

\begin{corollary}[Subject reduction for $\beta$]\label{finalbetaeta}
\mbox \hfill{}\\
If $M : \<\G \v  U\>$ and $M \rhd^*_{\beta} N$, then $N : \<\G\r_N
\v  U\>$.
\end{corollary}
\begin{proof}
By induction on the length of the derivation $M \rhd^*_{\beta} N$
using theorem~\ref{betaetatheo}.
\end{proof}

\begin{remark}
Note that using lemma~\ref{newgen}.(\ref{newgenfour} and~\ref{newgentwo}), we can also prove the
subject reduction property for  $\eta$-reduction.
\end{remark}
\subsection{Subject expansion for $\beta$}
Subject reduction for $\beta$ was shown using generation, substitution and environment restriction.  Subject expansion for $\beta$ needs something like the converse of the substitution lemma and environment enlargement.  

The next lemma can be seen as the converse of the substitution lemma.
\begin{lemma}\label{exp1}
If $M[x:= N] : \<\G \v U\>$, $x \in FV(M)$ and $x \not \in FV(N)$,
then $\exists~V$ type and $\exists~\G_1,\G_2$ type environments
such that:
\begin{itemize}
\item $M : \<\G_1, x : V \v U\>$
\item $N : \<\G_2 \v V\>$
\item $\G \sqsubseteq \G_1 \sqcap \G_2$
\end{itemize}
\end{lemma}

\begin{proof}
By induction on the derivation $M[x:= N] : \<\G \v U\>$.

If $M = x$, then $x : \<x : U \v U\>$, $N : \<\G \v U\>$ and $\G =
\G \sqcap ()$. Then we can assume that $M \neq x$.

\begin{itemize}
\item The last typing rule can not be $ax$.

\item
Let $\F{M[x:=N] : \<\G, y: W \v  T\>}{\l y. M[x:=N] : \<\G \v W \f
T\>}$ where $y \not \in FV(N)$.

By IH, $\exists~V$ type and $\exists~\G_1,\G_2$ type environments
such that $M : \<\G_1, x : V \v T\>$, $N : \<\G_2 \v V\>$ and $\G,
y : W \sqsubseteq \G_1 \sqcap \G_2$. Since $y \in FV(M)$ and $y
\not \in FV(N)$,  by lemma~~\ref{Phisub}.\ref{Phisubtwo},
$\G_1 = \D_1, y : W'$ and $W \sqsubseteq W'$. Hence $M : \<\D_1,
y : W', x : V \v T\>$. By rule $\f_i$, $\l y. M :
\<\D_1 , x : V \v W' \f T\>$ and since $W' \f T \sqsubseteq W \f T$, then by rule $\sqsubseteq$, $\l y. M
: \<\D_1 , x : V \v W \f T\>$. Finally
by lemma~\ref{Phisub}.\ref{Phisubtwo}, $\G \sqsubseteq \D_1 \sqcap \G_2$.

\item
Let $\F{M[x:=N] : \<\G \v  T\> \;\;\;y \not \in dom(\G)}{\l y.
M[x:=N]  : \<\G\v \omega \f T\>}$.

By IH, $\exists~V$ type and $\exists~\G_1,\G_2$ type environments
such that $M : \<\G_1, x : V \v T\>$, $N : \<\G_2 \v V\>$ and $\G
\sqsubseteq \G_1 \sqcap \G_2$. Since $y \neq x$, $\l y.M : \<\G_1,
x : V \v \o \f T\>$.

\item
Let $\F{M_1[x:=N] : \<\G_1 \v  W \f T\> \;\;\; M_2[x:=N] : \<\G_2
\v W\>} {M_1[x:=N] \; M_2[x:=N] : \<\G_1 \sqcap \G_2 \v T\>}$

where  $M = M_1 M_2$ and $x \in FV(M_1)\cap FV(M_2)$.

By IH, $\exists~V_1,V_2$ types and $\exists~\D_1,\D_2;\B_1,\B_2$
type environments such that $M_1 : \<\D_1, x : V_1 \v W \f T\>$,
$M_2 : \<\B_1, x : V_2 \v W\>$, $N : \<\D_2 \v V_1\>$, $N : \<\B_2
\v V_2\>$, $\G_1 \sqsubseteq \D_1 \sqcap \D_2$ and $\G_2
\sqsubseteq \B_1 \sqcap \B_2$. Then, by rules $\sqcap'$ and $\f_e$, $M_1 M_2:
\<\D_1 \sqcap \B_1, x : V_1  \sqcap V_2 \v T\>$ and $N :
\<\Delta_2 \sqcap \B_2 \v V_1 \sqcap V_2\>$. Finally, by lemma~\ref{Phisub}.\ref{Phisublast}, $\G_1 \sqcap
\G_2 \sqsubseteq (\D_1 \sqcap \D_2) \sqcap (\B_1 \sqcap \B_2)$.

The cases $x \in FV(M_1)\setminus FV(M_2)$ or $x \in
FV(M_2)\setminus FV(M_1)$ are easy.

\item
Let $\F{M[x:=N] : \<\G \v  U_1\> \;\; M[x:=N]  : \<\G \v U_2\>}
{M[x:=N] : \<\G  \v  U_1 \sqcap U_2\>}$.

By IH, $\exists~V_1,V_2$ types and $\exists~\G_1,\G_2;\D_1,\D_2$
type environments such that $M : \<\G_1, x : V_1 \v U_1\>$, $M :
\<\D_1, x : V_2 \v U_2\>$, $N : \<\G_2 \v V_1\>$, $N : \<\D_2 \v
V_2\>$, $\G \sqsubseteq \G_1 \sqcap \G_2$ and $\G \sqsubseteq \D_1
\sqcap \D_2$. Then, by rule $\sqcap'$, $M : \<\G_1 \sqcap \D_1, x
: V_1 \sqcap V_2 \v U_1 \sqcap U_2\>$ and $N : \<\G_2 \sqcap \D_2
\v V_1 \sqcap V_2\>$. Finally, by lemma~\ref{Phisub}.\ref{Phisublast},
$\G \sqsubseteq (\G_1 \sqcap
\G_2) \sqcap (\D_1 \sqcap \D_2)$.

\item
Let $\F{M[x:=N] : \<\G' \v U'\> \;\;\; \<\G'\v  U'\> \sqsubseteq
\<\G \v U\>}{M : \<\G \v U\>}$.

By lemma~\ref{Phisub}.\ref{Phisubtwo}, $\G \sqsubseteq \G'$ and
$U' \sqsubseteq U$. By IH, $\exists~V$ type and
$\exists~\G'_1,\G'_2$ type environments such that $M : \<\G_1, x :
V \v U'\>$, $N : \<\G_2 \v V\>$ and $\G' \sqsubseteq \G_1 \sqcap
\G_2$. Then by rules $\sqsubseteq_{\<\>}$, $\sqsubseteq$ and {\it tr},
$M : \<\G_1, x : V \v U\>$ and $\G \sqsubseteq \G_1
\sqcap \G_2$.

\end{itemize}
\end{proof}

Since more free variables might appear in the $\beta$-expansion of a term, the next definition gives a possible enlargement of an environment.

\begin{definition}
Let $m \geq n$, $\G = (x_i : U_i)_n$ and ${\cal U} =
\{x_1,...,x_m\}$. We write $\G {\u^{\cal U}}$ for 
$x_1 : U_1,...,x_n : U_n,x_{n+1} : \o,...,x_m : \o$.
If $dom(\G) \subseteq  FV(M)$,
we write $\G {\u^M}$
instead of $\G {\u^{FV(M)}}$.
\end{definition}
The next lemma is basic for the proof of subject expansion for $\beta$.
\begin{lemma}\label{exp2}
If $M[x:=N] : \<\G \v U\>$, $x \not \in FV(N)$ and ${\cal U} =
FV((\l x.M)N)$, then $(\l x.M)N :\<\G {\u^{\cal U}} \v U\>$.
\end{lemma}

\begin{proof}
We have three cases:
\begin{itemize}
\item If $U = \o$: By lemma~\ref{structyping}.\ref{structypingtwo},
we have $(\l x.M)N :\<\G {\u^{\cal U}} \v \o\>$.
\item If $U \in {\mathbb T}$: We have two cases:
\begin{itemize}
\item If $x \in FV(M)$, then, by lemma~\ref{exp1}, $\exists~V$ type and
$\exists~\G_1,\G_2$ type environments such that $M : \<\G_1, x : V
\v U\>$, $N : \<\G_2 \v V\>$ and $\G \sqsubseteq \G_1 \sqcap
\G_2$. Hence, by rules $\f_i$ and $\f_e$, $\l x. M : \<\G_1 \v V
\f U\>$ and $(\l x. M)N : \<\G_1 \sqcap \G_2 \v U\>$. Since $FV((\l
x.M)N) = FV(M[x:=N])$, then $\G {\u^{\cal U}} = \G$, and, by rule
$\sqsubseteq$, $(\l x.M)N :\<\G {\u^{\cal U}} \v U\>$.
\item If $x \not \in FV(M)$, then $M : \<\G \v U\>$ and, by rule
$\f'_i$, $\l y.M : \<\G \v \o \f U\>$. By rule $\o$, $N :
\<env_\o^N \v \o\>$, then, by rule $\f_e$, $(\l x.M)N :\<\G \sqcap
env_\o^N \v U\>$. Since  $FV((\l x.M)N) = FV(M[x:=N]) \cup FV(N)$,
then $\G {\u^{\cal U}} = \G \sqcap env_\o^N$.
\end{itemize}
\item If $U = \sqcap_{i=1}^k T_i$ where $\forall~1 \leq i \leq k$,
$T_i \in {\mathbb T}$: By rule $\sqsubseteq$, we have $\forall~1
\leq i \leq k$, $M[x:=N] : \<\G \v T_i\>$, then, by the previous
case, $\forall~1 \leq i \leq k$, $(\l x.M)N :\<\G {\u^{\cal U}} \v
T_i\>$, then, by $k-1$ applications of rule $\sqcap_i$, $(\l x.M)N :\<\G {\u^{\cal U}}
\v U\>$.
\end{itemize}
\end{proof}
Next, we give the main block for the proof of subject expansion for $\beta$.
\begin{theorem}\label{betaexp}
If $N : \<\G \v  U\>$ and $M \rhd_{\beta} N$, then $M :
\<\G {\u^{M}} \v U\>$.
\end{theorem}

\begin{proof}By induction on the derivation $N : \<\G \v  U\>$.
\begin{itemize}
\item
If $\F{}{x : \<x : T \v T\>}$ and $M \rhd_{\beta} x$, then $M = (\l y. M_1)M_2$ where $y \not \in FV(M_2)$ and $x = M_1[y:=M_2]$.  By lemma~\ref{exp2},
$M : \<(x:T) {\u^{M}} \v T\>$.
\item
If $\F{}{N : \<env^N_\o \v \o \>}$ and $M \rhd_{\beta} N$, then
since by theorem~\ref{conf}.\ref{confone}, $FV(N) \subseteq
FV(M)$, $(env^N_\o) {\u^{M}} = env^M_\o$.  By $\o$, $M : 
\<env^M_\o \v \o\>$.  Hence, $M  : \<(env^N_\o) {\u^{M}} \v \o\>$.
\item
If $\F{N : \<\G , x :U \v T\>}{\l x. N : \<\G \v U \f T\>}$ and $M \rhd_{\beta} \l x. N$, then we have two cases:
\begin{itemize}
\item
If $M = \l x.M'$ where $M' \rhd_{\beta} N$, then by IH, $M' :
\<(\G, x:U) {\u^{M'}} \v T\>$. Since by
theorem~\ref{conf}.\ref{confone} and
lemma~\ref{structyping}.\ref{structypingone}, $x \in FV(N)
\subseteq FV(M')$, then we have $(\G, x:U) {\u^{FV(M')}} =
\G {\u^{FV(M')\setminus\{x\}}}, x:U$ and
$\G {\u^{FV(M')\setminus\{x\}}} = \G {\u^{\l x.M'}}$. Hence,
 $M' : \<\G {\u^{\l x.M'}}, x:U \v T\>$ and finally, by $\f_i$, $\l x.M' :\<\G {\u^{\l x.M'}} \v U\f T\>$.
\item
 If $M = (\l y. M_1)M_2$ where $y \not \in FV(M_2)$
and $\l x. N = M_1[y:=M_2]$, then,
by lemma~\ref{exp2},
since $y \not \in FV(M_2)$ and $M_1[y:=M_2] : \<\G \v U \f T\>$,
we have $(\l y. M_1)M_2  : \<\G {\u^{(\l y. M_1)M_2}} \v U \f T\>$.
\end{itemize}
\item
If $\F{N : \<\G \v T\>\;\;\;x \not \in dom(\G)}{\l x. N : \<\G \v \o \f T\>}$ and $M \rhd_{\beta} N$  then similar to the above case.
\item
If $\F{N_1 : \<\G_1 \v U \f T\> \;\;\; \hspace{0.2in} N_2 : \<\G_2 \v
U\>}{N_1 \; N_2 : \<\G_1 \sqcap \G_2 \v T\>}$ and $M \rhd_{\beta} N_1N_2$,
we have three cases:
\begin{itemize}
\item
$M = M_1N_2$ where $M_1 \rhd_{\beta} N_1$. By IH, $M_1 : \<\G_1 {\u^{M_1}} \v U \f T\>$.  It is easy to show that $(\G_1 \sqcap \G_2) {\u^{M_1N_2}} = \G_1 {\u^{M_1}}\sqcap \G_2$. Now use $\f_e$.
\item
$M = N_1M_2$ where $M_2 \rhd_{\beta} N_2$. Similar to the above case.
\item
$M = (\l x.M_1)M_2$ where $x \not \in FV(M_2)$ and $N_1N_2 = M_1[x:=M_2]$.
By lemma~\ref{exp2}, $(\l x.M_1)M_2 :\<(\G_1 \sqcap \G_2) {\u^{(\l x.M_1)M_2}} \v T\>$.
\end{itemize}
\item
If $\F{N: \<\G \v U_1\> \;\;\;\hspace{0.2in} N : \<\G \v U_2\>}
{N : \<\G \v U_1 \sqcap U_2\>}$ and $M \rhd_{\beta} N$ then  use IH.
\item
Let $\F{N : \<\G\v U\> \;\;\;\hspace{0.2in} \<\G\v U\> \sqsubseteq \<\G'\v U'\>}
{N : \<\G'\v U'\>}$ and $M \rhd_{\beta} N$.  By
lemma~\ref{Phisub}.\ref{Phisubthree}, $\G' \sqsubseteq \G$ and $U \sqsubseteq U'$. It is easy to show that
$\G' {\u^{M}} \sqsubseteq \G {\u^{M}}$ and hence by  lemma~\ref{Phisub}.\ref{Phisubthree}, $\<\G {\u^{M}}\v U\> \sqsubseteq \<\G' {\u^{M}}\v U'\>$. By IH, $M {\u^{M}} : \<\G\v U\>$.  Hence, by $\sqsubseteq_{\<\>}$, we have $M : \<\G' {\u^{M}}\v U'\>$.
\end{itemize}\end{proof}

\begin{corollary}[Subject expansion for $\beta$]\label{finalbetaexp}
\mbox \hfill{}\\
If $N : \<\G \v  U\>$ and $M \rhd^*_{\beta} N$, then $M :
\<\G {\u^{M}} \v U\>$.
\end{corollary}

\begin{proof}
By induction on the length of the derivation $M \rhd^*_{\beta} N$
using theorem~\ref{betaexp} and the fact that if $FV(P) \subseteq FV(Q)$, then $(\G {\u^{P}}) {\u^{Q}} = \G {\u^{Q}}$.
\end{proof}

\section{The realisability semantics,  its soundness and completeness}
\label{realsem}
In this section we give a realisability semantics for our type system and establish both the soundness and completeness of this semantics.

We start with the definition of the function space and saturated sets.  

\begin{definition}  Let ${\cal X},{\cal Y} \subseteq {\cal M}$.
\begin{enumerate}
\item We use ${\cal P}({\cal X})$
to denote the powerset of ${\cal X}$, i.e.\ $\{{\cal Y} \; / \;
{\cal Y} \subseteq {\cal X}\}$.
\item We define ${\cal X} \fx {\cal Y} = \{M \in {\cal M}$ / $M \; N \in
{\cal Y}$ for all $N \in {\cal X}\}$.
\item Let $r \in \{f, \beta\}$.  We say that ${\cal X}$ is $r$-saturated if whenever
$M \rhd_r^* N$ and $N \in {\cal X}$, then $M \in
{\cal X}$.
\end{enumerate}
\end{definition}

\begin{lemma}\label{fx+} Let $r \in \{f, \beta\}$. 
\begin{enumerate}
\item \label{fx+one}
If ${\cal X}$ is $\be$-saturated, then ${\cal X}$ is $f$-saturated.
\item  \label{fx+two} If ${\cal X},{\cal Y}$ are $r$-saturated sets,
then ${\cal X} \cap {\cal Y}$ is  $r$-saturated.
\item  \label{fx+three} If ${\cal Y}$ is  $r$-saturated, then, for every set
${\cal X} \subseteq {\cal M}$, ${\cal X} \fx {\cal Y}$ is
 $r$-saturated.
\end{enumerate}
\end{lemma}

\begin{proof}
\ref{fx+one}.\ Note that $\rhd_f^*\subset \rhd_\be^*$.
\ref{fx+two}.\
is easy. 
\ref{fx+three}.\  Let $N \in {\cal X} \fx {\cal Y}$, $M \rhd_r^* N$
and $P \in {\cal X}$. Then, by theorem~\ref{conf}.\ref{conftwo},
$M \; P \rhd_r^* N \; P$ and $N \; P \in {\cal Y}$. Since ${\cal Y}$
is  $r$-saturated, then $M \; P \in {\cal Y}$. Thus, $M \in {\cal X}
\fx {\cal Y}$.
\end{proof}

We interpret basic types as saturated sets. The interpretation of complex types  is built up from smaller types in the obvious way.
\begin{definition}\label{defint}
 Let $r \in \{f, \beta\}$. 
\begin{enumerate}
\item\label{defintone}
An  $r$-interpretation ${\cal I}: {\cal A} \mapsto {\cal P}({\cal M})$
is  a function 
such that: \\ $\forall~a \in {\cal A}$, ${\cal I}(a)$
is $r$-saturated.
\item\label{definttwo}
An $r$-interpretation ${\cal I}$ can be extended to ${\mathbb U}$ as follows:\\
${\cal I}(\omega) = {\cal M}$ \hspace{0.19in} $ {\cal I}(U_1 \sqcap U_2) = {\cal I}(U_1) \cap {\cal I}(U_2)$
\hspace{0.19in} ${\cal I}(U \f T) = {\cal I}(U) \fx
{\cal I}(T)$
\end{enumerate}
\end{definition}
\begin{lemma}
\label{beintiswint}
If ${\cal I}$ is a $\be$-interpretation then  ${\cal I}$ is an $f$-interpretation.
\end{lemma}
\begin{proof}
Use lemma~\ref{fx+}.\ref{fx+one}.
\end{proof}
The next lemma shows that the interpretation of any type (basic or complex) is saturated, that the interpretation function respects the relation $\sqsubseteq$ and that we can in some sense expand the terms in the interpretation.
\begin{lemma}\label{combinedlem}
 Let $r \in \{f, \beta\}$ and let ${\cal I}$ be an $r$-interpretation.
\begin{enumerate}
\item\label{interpret}
For any  $U\in  {\mathbb U}$, we have ${\cal I}(U)$ is $r$-saturated.
\item
\label{intsub}
If $U \sqsubseteq V$, then
${\cal I}(U) \subseteq {\cal I}(V)$.
\item \label{neededcombine} Let $n \geq 0$ and $\forall 1 \leq i \not = j \leq n$, $x_i \not = x_j$.
If $\forall~N_i \in {\cal I}(U_i)$ ($1 \leq i \leq n$), $M[(x_i :=
N_i)_1^n] \in {\cal I}(U)$, then  \\$\l x_1....\l x_n .M \in {\cal I}(U_1 \f (U_2 \f (...\f
(U_n \f U)...)))$.

\end{enumerate}
\end{lemma}

\begin{proof}
\ref{interpret}.\ By induction on $U$ using lemma~\ref{fx+}.\\
\ref{intsub}.\
By induction of the derivation $U \sqsubseteq V$.
\ref{neededcombine}.\ By induction on $n \geq 0$ using \ref{interpret}.
\end{proof}

We now show the soundness of our sematics.
\begin{theorem}[Soundness]\label{adeq}
Let $r \in \{f, \be\}$.
If $M : \<(x_i:U_i)_n \v U\>$, ${\cal I}$ is an $r$-interpretation and
$\forall 1 \leq i \leq n$, $N_i \in {\cal I}(U_i)$, then
$M[(x_i:=N_i)_1^n] \in {\cal I}(U)$.
\end{theorem}

\begin{proof}
By induction on the derivation  $M : \<(x_i:U_i)_n \v U\>$.

\begin{itemize}
\item
Let $\F{}{x : \<(x:T) \v T\>}$.  If $N \in {\cal I}(T)$ then
$x[x:=N] = N \in {\cal I}(T)$.

\item
Let $\F{}{M : \<env_\o^{M} \v \o\>}$ where $env_\o^{M} =
(x_i:\omega)_n$.

We have $M[(x_i:=N_i)_1^n] \in {\cal M} = {\cal I}(\omega)$.

\item Let
$\F{P : \< (x_i:U_i)_1^n ,x: U \v T\>} {\l x. P : \< (x_i:U_i)_n\v
U \f T\>}$.  \\
If ${\cal I}(U)=\emptyset$ then $(\l x.
P)[(x_i:=N_i)_1^n] \in {\cal I}(U) \fx {\cal I}(T) = {\cal M}$.\\
 If ${\cal I}(U)\not =\emptyset$ then  let $N \in {\cal I}(U)$.
By IH, $P[(x_i:=N_i)_1^n,x:=N] \in {\cal I}(T)$. By
lemma~\ref{combinedlem}.\ref{interpret}, ${\cal I}(T)$ is
$r$-saturated. \\Moreover, $(\l x.P)[(x_i:=N_i)_1^n] \; N \rhd_r^*
P[(x_i:=N_i)_1^n,x:=N]$.  Hence,\\
$(\l x.P)[(x_i:=N_i)_1^n] N \in {\cal I}(T)$ and $(\l x.
P)[(x_i:=N_i)_1^n] \in {\cal I}(U) \fx {\cal I}(T)$.

\item Let
$\F{P : \< (x_i:U_i)_n \v T\>\;\;\;x \neq x_i} {\l x. P : \<
(x_i:U_i)_n\v \omega \f T\>}$ and $N \in {\cal M}$.
Note that $x \not \in FV(P)$.

By IH, $P[(x_i:=N_i)_1^n] \in {\cal
I}(T)$. By lemma~\ref{combinedlem}.\ref{interpret}, ${\cal I}(T)$ is $r$-saturated.\\
Moreover, $(\l x. P)[(x_i:=N_i)_1^n] \; N \rhd_r^*
P[(x_i:=N_i)_1^n]$.
Hence\\
$(\l x.P)[(x_i:=N_i)_1^n] \; N \in {\cal I}(T)$ and $(\l x.
P)[(x_i:=N_i)_1^n] \in {\cal I}(\o) \fx {\cal I}(T)$.

\item Let  $\F{M_1 : \<\G_1 \v U \f T\>
\;\;\; M_2 : \<\G_2 \v U\>} {M_1 \; M_2 : \<\G_1 \sqcap \G_2
 \v T\>}$ where
$\G_1 = (x_i:U_i)_n, (y_j:V_j)_m$, $\G_2 =
(x_i:U'_i)_n,(z_k:W_k)_l$ and $\G_1 \sqcap \G_2 = (x_i:U_i
\sqcap U'_i)_n,(y_j:V_j)_m,(z_k:W_k)_l$.
Let $\forall 1 \leq i \leq n, P_i \in {\cal I}(U_i \sqcap U'_i)$,
$\forall 1 \leq j \leq m, Q_j \in {\cal I}(V_j)$ and $\forall 1
\leq k \leq l, R_k \in {\cal I}(W_k)$.
By IH, $M_1[(x_i:=P_i)_1^n,(y_j:=Q_j)_1^m] \in {\cal I}(U) \fx
{\cal I}(T)$ and

$M_2[(x_i:=P_i)_1^n,(z_k:=R_k)_1^l] \in {\cal I}(U)$,

then $(M_1 M_2)[(x_i:=P_i)_1^n,(y_j:=Q_j)_1^m,(z_k:=R_k)_1^l] =$

$M_1[(x_i:=P_i)_1^n,(y_j:=Q_j)_1^m] \; M_2[(x_i:=P_i)_1^n,(z_k:=R_k)_1^l]
\in {\cal I}(T)$.

\item
Let $\F{M: \<(x_i:U_i)_n \v V_1\> \;\;\; M : \<(x_i:U_i)_n \v
V_2\>} {M : \<(x_i:U_i)_n \v V_1 \sqcap V_2\>}$.
By IH, $M[(x_i:=N_i)_1^n] \in {\cal I}(V_1)$ and $M[(x_i:=N_i)_1^n]
\in {\cal I}(V_2)$. Hence, $M[(x_i:=N_i)_1^n] \in {\cal I}(V_1
\sqcap V_2)$.

\item
Let $\F{M : \Phi \;\;\; \Phi \sqsubseteq \Phi'} {M : \Phi'}$ where
$\phi' = \<(x_i:U_i)_n \v U\>$.

By lemma \ref{Phisub}.\ref{Phisubthree} and
\ref{Phisub}.\ref{Phisubtwo}, $\Phi = \<(x_i : U'_i)_n \v U'\>$,
$\forall~1 \leq i \leq n$, $U_i \sqsubseteq U'_i$ and $U'
\sqsubseteq U$. By lemma \ref{combinedlem}.\ref{intsub}, $N_i \in {\cal I}(U'_i)$,
then, by IH, $M[(x_i:=N_i)_1^n] \in {\cal I}(U')$ and, by lemma
\ref{combinedlem}.\ref{intsub}, $M[(x_i:=N_i)_1^n] \in {\cal I}(U)$.
\end{itemize}
\end{proof}

Roughly speaking, completeness of the semantics amounts to saying that 
if $M$ is in the meaning of type $U$ (i.e., $M$ is in ${\cal I}(U)$ for any interpretation ${\cal I}$) then $M$ has type $U$.  In order to show completeness, we define a special interpretation function ${\mathbb I}$ through the typing relation $\v$ in such a way that, if $M \in {\mathbb I}(U)$ then $M$ can be shown to have type $U$. This is done in the next definition and lemma.  

\begin{definition}
\label{finalintfunc}
\begin{enumerate}
\item \label{finalintfuncone} For every  $U \in {\mathbb U}$, let
an infinite subset ${\mathbb V}_U$ of ${\cal V}$ such that:\\
$\bullet$ If $U \neq V$, then ${\mathbb V}_U
\cap {\mathbb V}_V = \emptyset$. \hspace{0.5in}
$\bullet$
$\bigcup_{U \in {\mathbb U}} {\mathbb V}_U = {\cal V}$.
\item\label{finalintfunctwo} We denote ${\mathbb G}=\{(x:U)$ / $U$ is a type and $x \in
{\mathbb V}_U \}$. Note that since ${\mathbb G}$ is infinite,  ${\mathbb G}$ is not a type
environment.
\item\label{finalintfuncthree} Let $M \in{\cal M}$ and $U \in {\mathbb U}$.  We write
 $M : \< {\mathbb G} \v U\>$ if there is a type
environment $\G \subset {\mathbb G}$ such that $M : \< \G \v U\>$.
\item \label{finalintfuncfour}
Let ${\mathbb I}: {\cal A} \mapsto {\cal P}({\cal M})$ be the 
function defined by:

$\forall~a \in {\cal A}$, ${\mathbb I}(a) = \{M \in {\cal M}$ / $M
: \< {\mathbb G} \v a\> \}$.
\end{enumerate}
\end{definition}
\begin{remark}
\label{useremark}
Note that in Definition~\ref{finalintfunc},  we have associated to each $U \in {\mathbb U}$, an infinite set of variables ${\mathbb V}_U$ in such a way that no variable is used in two different types, and each variable of ${\cal V}$ is associated to a type. Obviously,  as long as these conditions are satisfied, we have the liberty of dividing the set  ${\cal V}$ as we wish.  We will practice this liberty in the proof of theorem~\ref{comple4}.
\end{remark}

\begin{lemma}\label{comp}
\begin{enumerate}
\item\label{compone'}
If $\G, \G' \subset {\mathbb G}$ and $dom(\G) = dom(\G')$, then
$\G = \G'$.
\item\label{compone''}
If $\G, \G' \subset {\mathbb G}$, then $\G \sqcap \G' = \G \cup
\G' \subset{\mathbb G}$.
\item\label{compone}
 ${\mathbb I}$ is a $\be$-interpretation.  I.e., $\forall~a \in {\cal A}$, ${\mathbb I}(a)$ is $\be$-saturated.\\
Hence, ${\mathbb I}$ is an $f$-interpretation. \\
Furthermore, we extend  ${\mathbb I}$ to ${\mathbb U}$ as in Definition~\ref{defint}.\ref{definttwo}.
\item\label{comptwo}
If $U \in {\mathbb U}$, then ${\mathbb I}(U) \not = \emptyset$ and 
${\mathbb I}(U) = \{M \in {\cal M}$ /
$M : \<{\mathbb G} \v U\> \}$.
\end{enumerate}
\end{lemma}
\begin{proof}
\begin{itemize}
\item[\ref{compone'}.]
Let $(x:U) \in \G$ and $(x:U') \in \G'$.  Hence, $x \in {\mathbb V}_U$ and $x \in {\mathbb V}_{U'}$
and so, $U = U'$ (otherwise, ${\mathbb V}_U \cap {\mathbb V}_{U'} = \emptyset$).\item[\ref{compone''}.] Let $\G = (x_i:U_i)_n, (y_j:V_j)_m$ and
$\G' = (x_i:U'_i)_n, (z_k:W_k)_l$ where $y_j \not = z_k$ for all $1\leq j \leq m$ and $1\leq k \leq l$.  Since
$(x_i:U_i)_n\subset {\mathbb G}$ and $(x_i:U'_i)_n\subset {\mathbb G}$, by
\ref{compone'}, $U_i = U'_i$ for all $1 \leq i \leq n$.  Hence, $\G \sqcap \G' = \G \cup \G' \subset{\mathbb G}$. 
\item[\ref{compone}.] Let $a \in {\cal A}$, $M  \in {\cal M}$, $M \rhd_\be^*
N$ and $N \in {\mathbb I}(a)$.  Then $N: \<\G \v a\>$ where $\G
\subset {\mathbb G}$. Let $FV(M) \setminus dom(\G) = \{x_1,...,x_n
\}$ and  $\forall~1 \leq i \leq n$, take $U_i$  such that $x_i \in
{\mathbb V}_{U_i}$.  Then  $\D = \G,(x_i:U_i)_n \subset {\mathbb
G}$ and  $\G {\u^M} = \G , (x_i : \o)_n$. By
corollary~\ref{finalbetaexp}, $M : \<\G {\u^M} \v a\>$ and, by
lemma~\ref{Phisub}.\ref{Phisubtwo}, $\D \sqsubseteq \G {\u^M}$.
Hence, by rule $\sqsubseteq$,  $M : \<\D \v a\>$.
Thus, $M \in {\mathbb I}(a)$. Hence ${\mathbb I}(a)$ is $\be$-saturated and 
so,  ${\mathbb I}$ is a $\be$-interpretation.
 Finally, by lemma~\ref{beintiswint},  ${\mathbb I}$ is an $f$-interpretation.
\item[\ref{comptwo}.]
The proof of  ${\mathbb I}(U) \not = \emptyset$ is as follows:
let $x \in {\mathbb V}_{U}\not = \emptyset$.  Then, $x:U \in  {\mathbb G}$ and since
$x:\<(x:U)\v U\>$ then $x\in{\mathbb I}(U)$.  

Now we do the second part by induction on $U$.
\begin{itemize}
\item $U = a$: By definition of ${\mathbb I}$.
\item $U = \o$:
By definition,
 ${\mathbb I}(\o) = {\cal M}$. So,
$\{M \in {\cal M}$ / $M : \<{\mathbb G} \v \o\> \} \subseteq
{\mathbb I}(\o)$.\\
Conversely, let $M \in {\mathbb I}(\o)$ where $FV(M) = \{x_1,...,x_n\}$. We
have $M : \<(x_i:\o)_n \v \o\>$. $\forall~1 \leq i \leq n$, take
$U_i$ such that $x_i \in {\mathbb V}_{U_i}$.  Then  $\G =
(x_i:U_i)_n \subset {\mathbb G}$. By lemma~\ref{structyping}.\ref{structypingtwo}, $M :
\<\G \v \o\>$. Hence $M : \<{\mathbb G} \v \o\>$. Thus, ${\mathbb
I}(\o) \subseteq \{M \in {\cal M}$ / $M : \< {\mathbb G} \v \o\>
\}$.

We deduce ${\mathbb I}(\o) = \{M \in {\cal M}$ / $M : \< {\mathbb
G} \v \o\> \}$.

\item $U = U_1 \sqcap U_2$:
By IH, ${\mathbb I}(U_1 \sqcap U_2) = {\mathbb I}(U_1) \cap
{\mathbb I}(U_2) = $

$\{M \in {\cal M}$ / $M : \< {\mathbb G} \v U_1\> \} \cap \{M \in
{\cal M}$ / $M : \< {\mathbb G} \v U_2\> \}$.
\begin{itemize}
\item If $M : \< {\mathbb G} \v U_1\>$
and $M : \< {\mathbb G} \v U_2\>$, then $M : \< \G_1 \v U_1\>$ and
$M : \< \G_2 \v U_1\>$ where $\G_1,\G_2 \subset {\mathbb G}$. By
lemma~\ref{structyping}.\ref{structypingone}, $dom(\G_1) =
dom(\G_2) = FV(M)$.  By lemma~\ref{newrules}.\ref{newrulesone}, $M
: \< \G_1 \sqcap \G_2 \v  U_1 \sqcap U_2 \>$. Since $\G_1,\G_2
\subset {\mathbb G}$, then, by~\ref{compone'}, $\G_1 = \G_2$ and
$\G_1 \sqcap \G_2 = \G_1 \subset {\mathbb G}$. Thus $M : \<
{\mathbb G} \v U_1 \sqcap U_2\>$.
\item If $M : \< {\mathbb G} \v U_1 \sqcap
U_2\>$, then $M : \< \G \v  U_1 \sqcap U_2\>$ where $\G \subset
{\mathbb G}$. By $\sqsubseteq$, $M : \< \G \v U_1\>$ and $M : \<
\G \v  U_2\>$, then $M : \< {\mathbb G} \v U_1\>$ and $M : \<
{\mathbb G} \v U_2\>$.
\end{itemize}
We deduce ${\mathbb I}(U_1 \sqcap U_2) = \{M \in {\cal M}$ / $M :
\< {\mathbb G} \v U_1 \sqcap U_2\> \}$.

\item $U = V \f T$: Then ${\mathbb I}(V \f T) = {\mathbb I}(V) \fx {\mathbb I}(T)$.
By IH, \\${\mathbb I}(V) = \{M \in {\cal M}$ / $M : \< {\mathbb G}
\v V\> \}$ and $ {\mathbb I}(T) = \{M \in {\cal M}$ / $M : \<
{\mathbb G} \v T\>\}$.
\begin{itemize}
\item Let $M \in {\mathbb I}(V) \fx {\mathbb I}(T)$ and $x \in {\mathbb V}_V$
such that $x \not \in FV(M)$. By rule $ax'$ (see
lemma~\ref{newrules}.\ref{newrulestwo}), $x : \<(x : V) \v V\>$.
Since $(x:V) \subset {\mathbb G}$, then $x : \<{\mathbb G} \v
V\>$. By IH, $x \in {\mathbb I}(V)$. Hence $M x \in {\mathbb
I}(T)$ and so $M  x : \< \G \v T\>$ where $\G \subset {\mathbb
G}$. Since $x \not \in FV(M)$, then $\G = \D , x : V$ and $\D
\subset {\mathbb G}$. By lemma~\ref{newgen}.\ref{newgenfour}, we
deduce that $M : \< \D \v V \f T\>$.
\item Let $M,N \in {\cal M}$ such that $M : \<{\mathbb G} \v V \f T\> $
and $N: \< {\mathbb G} \v V\> $. We have $M : \< \G_1 \v V \f T\>$
and $N : \< \G_2 \v V\>$ where $\G_1,\G_2 \subset {\mathbb G}$.
Thus $M \; N : \< \G_1 \sqcap \G_2 \v T\>$.  Since, by
lemma~\ref{comp}.\ref{compone''}$, \G_1 \sqcap \G_2 \subset
{\mathbb G}$. Therefore $M  N : \< {\mathbb G} \v T\>$.
\end{itemize}

We deduce ${\mathbb I}(V \f T) = \{M \in {\cal M}$ / $M : \<
{\mathbb G} \v V \f T\> \}$.
\end{itemize}
\end{itemize}
\end{proof}

Now, the ${\mathbb I}$ of definition~\ref{finalintfunc}
will be used to show the completeness of the semantics.
\begin{theorem}[Completeness]\label{completude}
Let $r \in \{f, \be\}$.  
Let $U_1,...,U_n,U \in {\mathbb U}$  and $M \in {\cal M}$ such
that $FV(M) = \{x_1,...,x_n\}$. If $\forall$ $r$-interpretation ${\cal
I}$ and $\forall~N_i \in {\cal I}(U_i)$ ($1 \leq i \leq n$),
$M[(x_i := N_i)_1^n] \in {\cal I}(U)$, then $M : \< (x_i : U_i)_n
\v U\>$.
\end{theorem}

\begin{proof}
We distinguish three cases:
\begin{itemize}
\item If $U = \o$, then  $M : \< (x_i : \o)_n \v
\o\>$. Thus, by lemma~\ref{structyping}.\ref{structypingtwo}, $M : \< (x_i : U_i)_n \v
\o\>$.
\item If $U \in {\mathbb T}$, then, let $V = U_1 \f (U_2 \f (...\f
(U_n \f U)...))$. By hypothesis and lemma~\ref{combinedlem}.\ref{neededcombine},
 $\forall$
$r$-interpretation ${\cal I}$, $\l x_1....\l x_n .M \in {\cal I}(V)$.
Hence, $\l x_1....\l x_n .M \in {\mathbb I}(V)$ where ${\mathbb
I}$ is the interpretation of
definition~\ref{finalintfunc}.\ref{finalintfuncfour}. By
lemma~\ref{comp}.\ref{comptwo},  $\l x_1....\l x_n .M: \< \G \v
V\>$ where $\G \subset {\mathbb G}$ and, since $\l x_1....\l x_n.
M$ is closed, $\G = ()$. By rule $ax'$, $\forall~1 \leq i \leq n$,
$x_i : \<x_i : U_i \v U_i\>$, by $n$ applications of $\f_e$ we
deduce $(\l x_1....\l x_n.M) x_1...x_n: \< (x_i : U_i)_n \v U\>$.
Since $(\l x_1....\l x_n.M) x_1...x_n \rb^* M$, then by
corollary~\ref{finalbetaeta}, $M : \< (x_i : U_i)_n \v U\>$.
\item If $U = \sqcap_{j=1}^m T_j$, then, by hypothesis, $\forall$
$r$-interpretation ${\cal I}$, $\forall~N_i \in {\cal I}(U_i)$ ($1 \leq i \leq
n$), and $\forall~1 \leq j \leq m$, $M[(x_i := N_i)_1^n] \in {\cal
I}(T_j)$. By the previous case, $\forall~1 \leq j \leq m$, $M : \<
(x_i : U_i)_n \v T_j\>$.  By $m-1$ applications of $\sqcap_i$ we
deduce $M : \< (x_i : U_i)_n \v U\>$.
\end{itemize}
\end{proof}
\section{The meaning of types}
\label{meantypes}
Obviously the meaning of a type $U$ should be based on the intersection of 
all the interpretations of $U$.   However, since we have been using two different kinds of interpretations ($\be$- and $f$-interpretations), we give two definitions for  the meaning of a type.  We will show that these two definitions are equivalent.

\begin{definition}
\label{meaningopen}
Let $r \in \{f, \beta\}$. 
We define the meaning $[U]_r$ of $U\in {\mathbb U}$ by:

$$[U]_r=  \bigcap_{{\cal I} \;\;\; r-interpretation} {\cal I}(U)$$
\end{definition}
The next theorem shows that the meaning $[U]$ of $U$ is the set of terms typable by $U$ in a special environment and that $[U]$ is stable by $\beta$-reduction and $\beta$-expansion.
\begin{theorem}\label{comple4}
Let $r \in \{f, \beta\}$ and  $U\in{\mathbb U}$.
\begin{enumerate}
\item  \label{comple4one} $[U]_r = \{M \in {\cal M}$ / $M : \< env_\o^M \v  U\> \}$.
\item  \label{comple4two} $[U]_r$ is stable by $\beta$-reduction.  I.e., if $M \in [U]_r$ and $M \rb^* N$,
then $N \in [U]_r$.
\item  \label{comple4three} $[U]_r$ is stable by $\beta$-expansion.
I.e., if $M \in [U]_r$, $N \rb^* M$, then $N \in [U]_r$.
\item \label{comple4one'} $[U]_r = \{M \in {\cal M}$ / $M \rb^* N$ and $N : \< env_\o^N \v  U\> \}$.
\end{enumerate}
\end{theorem}
\begin{proof}
\begin{itemize}
\item[\ref{comple4one}.]
Let $M \in {\cal M}$ such that $M : \< env_\o^M \v  U\>$. Let ${\cal I}$ be an $r$-interpretation and take  $FV(M) = dom(env_\o^M) = \{x_1, x_2, \dots, x_n\}$.  By theorem~\ref{adeq}, since $\forall 1 \leq i \leq n$, $x_i \in {\cal I}(\o) = {\cal M}$, then 
$M = M[(x:=x_i)_1^n] \in {\cal I}(U)$.  Hence, $M \in [U]_r$.\\
Conversely, let $M \in [U]_r$.  Take the interpretation ${\mathbb I}$ given in Definition~\ref{finalintfunc}
such that (recall remark~\ref{useremark}) $FV(M)\subset {\mathbb V}_\o$.
Since $M \in  {\mathbb I}(U)$ then 
$M : \< \G \v U \>$ where $\G \subseteq {\mathbb G}$.   But  $FV(M)\subset {\mathbb V}_\o$ and 
by lemma~\ref{structyping}.\ref{structypingone}, $FV(M) = dom(\G)$.
Hence $\G = env_\o^M$.\\
We conclude that 
 $[U]_r = \{M \in {\cal M}$ / $M : \< env_\o^M \v  U\> \}$.
\item[\ref{comple4two}.]
Let $M \in [U]_r$ such that $M \rb^* N$. By~\ref{comple4one}, $M :
\< env_\o^M \v U\>$. By
subject reduction for
$\beta$ corollary~\ref{finalbetaeta}, $N : \< (env_\o^M)\r_{N} \v
U\>$. Since by theorem~\ref{conf}.\ref{confone}, $FV(N) \subseteq FV(M)$ then $(env_\o^M)\r_{N} = env_\o^{N}$.
Thus by \ref{comple4one}, $N \in [U]_r$.
\item[\ref{comple4three}.]
Let $M \in [U]_r$ such that $N \rb^* M$. By~\ref{comple4one}, $M
: \< env_\o^M \v U\>$.  By subject expansion for 
$\beta$ corollary~\ref{finalbetaexp},
$N
: \< (env_\o^M)\u^N \v U\>$. Since by theorem~\ref{conf}.\ref{confone}, $FV(M) \subseteq FV(N)$ then $(env_\o^M)\u^{N} = env_\o^{N}$.
Thus by \ref{comple4one}, $N \in [U]_r$.
\item[\ref{comple4one'}.] 
By~\ref{comple4one}, $[U]_r \subseteq \{M \in {\cal M}$ / $M \rb^* N$ and $N : \< env_\o^N \v  U\> \}$. Conversely, let $M \rb^* N$ and $N : \< env_\o^N \v  U\>$. By~\ref{comple4one}, $N \in [U]_r$.  Hence, by~\ref{comple4three}, $M \in [U]_r$.
\end{itemize}
\end{proof}
\begin{corollary} Let $U \in {\mathbb U}$.  We have that 
$[U]_f = [U]_\be$.
\end{corollary}
\begin{proof}
By theorem~\ref{comple4}.\ref{comple4one}, 
$[U]_f = [U]_\be = \{M \in {\cal M}$ / $M : \< env_\o^M \v  U\> \}$.
\end{proof}
Hence, we write $[U]$ instead of either $[U]_f$ or $[U]_\be$.
\begin{remark}
\label{remarkme}
The reader may ask here why we introduced the two notions of saturation if the meaning of a type does not depend on whether this meaning was made using 
$\be$-interpretations or $f$-interpretations.  The answer to this question is that up to here, we could equally use $\be$-interpretations or $f$-interpretations.
However, to establish further results related to the meaning of types, especially for those types whose meaning consists of terms that reduce to closed terms, then we need  $\be$-saturation.
For this reason, in the rest of paper, we only
consider $\be$-saturation.
\end{remark}
Let us now 
reflect further on the meaning of types as given in definition~\ref{meaningopen}.  The next lemma gives three examples.
\begin{lemma}\label{example}
Let $a \in {\cal A}$, $U = \o \f (a \f a)$, $V = a \f (\o \f a)$
and \\$W = (\o \f a) \f a$. We have:
\begin{enumerate}
\item $[U] = \{ M \in {\cal M} / M \rb^* \l x. \l y. y\}$. Note that  $\l x. \l y. y : \<() \v U\>$.
\item $[V] = \{ M \in {\cal M} / M \rb^* \l x. \l y. x\}$. Note that $\l x. \l y. x : \<() \v V\>$.
\item $[W] = \{ M \in {\cal M} / M \rb^* \l x. x P$ where $P \in
{\cal M} \}$. \\ Note that $\l x. x P: \<env_\o^{\l x. x P} \v W\>$.
\end{enumerate}
\end{lemma}

\begin{proof}
\begin{enumerate}
\item It is easy to show that $\l x. \l y. y : \<() \v U\>$.  
Note that $env^{\l x. \l y. y}_\o  = ()$.  Hence, $\{ M \in {\cal M} / M \rb^* \l x. \l y. y\} = 
\{ M \in {\cal M} / M \rb^* \l x. \l y. y \mbox{ and } \l x. \l y. y : \< env^{\l x. \l y. y}_\o \v U\>\} \subseteq [U]$ by theorem~\ref{comple4}.\ref{comple4one'}.

Conversely, let $M \in [U]$ and $y \not \in FV(M)$. Take the $\be$-interpretation ${\cal I}$ such that ${\cal I}(a) = {\cal X} = \{M \in
{\cal M} / M \rb^* y\}$.  Since  $M \in [U]$ then $M \in {\cal I}(U) =  {\cal M} \fx ({\cal I}(a)  \fx {\cal I}(a)) =
{\cal M} \fx ({\cal X} \fx {\cal X})$.
Let $x \neq y$ such that $x \not \in
FV(M)$. Since $x \in {\cal M}$ and $y \in {\cal X}$, then $M 
xy \in {\cal X}$, $M x y \rb^* y$ and by 
theorem~\ref{conf}.\ref{conftwo'}, $M \rb^* \l x. \l y. y$.

\item  It is easy to show that  $\l x. \l y. x : \<() \v V\>$.
Let ${\cal I}$ be a $\be$-interpretation. By theorem~\ref{adeq}, $\l
x. \l y. x \in {\cal I}(V)$.  By
lemma~\ref{combinedlem}.\ref{interpret}, ${\cal I}(V)$ is
$\be$-saturated.  Hence, $\{ M \in {\cal M} / M \rb^* \l x. \l
y. x\}\subseteq {\cal I}(V)$.  Thus, $\{ M \in {\cal M} / M \rb^* \l
x. \l y. x\} \subseteq [V]$.\\ Conversely, let $M \in [V]$ and $x \not
\in FV(M)$.  Take the $\be$-interpretation ${\cal I}$ such that ${\cal
I}(a) = {\cal X} = \{M \in {\cal M} / M \rb^* x\}$.  Since $M \in [V]$
then $M \in {\cal I}(V) = {\cal I}(a) \fx ({\cal M} \fx {\cal I}(a)) =
{\cal X} \fx ({\cal M} \fx {\cal X})$.  Let $y \neq x$ such that $y
\not \in FV(M)$. We have $x \in {\cal X}$ and $y \in {\cal M}$, then
$M xy \in {\cal X}$ and $M x y \rb^* x$. Thus, by
theorem~\ref{conf}.\ref{conftwo'}, $M \rb^* \l x. \l y. x$.
\item Let $P \in {\cal M}$.  Using lemma~\ref{structyping}.\ref{structypingtwo}, we can show that  
$\l x. x P: \<env_\o^{\l x. x P} \v W\>$ (irrespectively of whether $x \in FV(P)$ or not).
 Now, $\{ M \in {\cal M} / M \rb^* \l x. xP\} = 
\{ M \in {\cal M} / M \rb^* \l x. xP \mbox{ and } \l x. xP : \< env^{\l x. xP}_\o \v W\>\} \subseteq [W]$ by theorem~\ref{comple4}.\ref{comple4one'}.

Conversely, let $M \in [W]$ and $x \not \in
FV(M)$. Take the $\be$-interpretation ${\cal I}$ such that ${\cal I}(a) = {\cal X} = \{M \in {\cal M} / M \rb^* x P$ where $P \in {\cal M}
\}$.  Then $M \in {\cal I}(W) = ({\cal M} \fx {\cal X}) \fx {\cal X}$.
Since $x \in {\cal M} \fx {\cal X}$, then $M \; x \in
{\cal X}$ and $M \; x \rb^* x P$ where $P \in {\cal M}$. Thus, by
theorem~\ref{conf}.\ref{conftwo'}, $M \rb^* \l x. x Q$ where $Q
\in {\cal M}$.
\end{enumerate}
\end{proof}

The meanings of the types $U$ and $V$ (of 
lemma~\ref{example}) contain only terms which are reduced to
closed terms.  Due to the position of $\o$ in $W$, the meaning 
of $W$ does not solely contain terms which are reduced to
closed terms. 
In $U$ and $V$, $\o$ has a negative occurence, but in
$W$, $\o$ has a positive one. We will generalize this
result.

\begin{definition}
\begin{enumerate}
\item We define two subsets  ${\mathbb U}^+$ and ${\mathbb
U}^-$ of ${\mathbb U}$ as follows:
\begin{itemize}
\item $\forall~a \in {\cal A}$, $a \in {\mathbb U}^+$ and $a \in {\mathbb
U}^-$.
\item $\o \in {\mathbb U}^-$.
\item If $U \in {\mathbb U}^+$, then $U \sqcap V \in {\mathbb
U}^+$.
\item If $U,V \in {\mathbb U}^-$, then $U \sqcap V \in {\mathbb
U}^-$.
\item If $U  \in {\mathbb U}^-$ and $T \in {\mathbb
U}^+$, then $U \f T \in {\mathbb U}^+$.
\item If $U  \in {\mathbb U}^+$ and $T \in {\mathbb
U}^-$, then $U \f T \in {\mathbb U}^-$.
\end{itemize}
\item Let ${\cal S} \subseteq {\cal V}$ where  ${\cal S} \not = \emptyset$.
\begin{enumerate}
\item We say that a term $M$ is ${\cal S}$-almost closed if $M \rb^* N$
and $FV(N) \subseteq {\cal S}$. We denote ${\cal M}^{\cal S}$ the
set of ${\cal S}$-almost closed terms.
\item We define the function 
${\cal I}_{\cal S} : {\cal A} \mapsto {\cal P}({\cal M})$ by:
$\forall~a \in {\cal A}$, ${\cal I}_{\cal S}(a) = {\cal M}^{\cal
S}$.
\end{enumerate}
\end{enumerate}
\end{definition}
The next lemma shows that ${\cal I}_{\cal S}$ is a $\be$-interpretation and 
relates 
${\cal I}_{\cal S}(U)$ and ${\cal M}^{\cal S}$ according to whether $U \in{\mathbb U}^+$ or $U \in{\mathbb U}^-$.
\begin{lemma}\label{key}
Let ${\cal S} \subseteq {\cal V}$ where  ${\cal S} \not = \emptyset$.
\begin{enumerate}
\item\label{keyzero}
${\cal I}_{\cal S}$ is a $\be$-interpretation.  I.e., $\forall~a \in {\cal A}$, ${\cal I}_{\cal S}(a)$ is $\be$-saturated.\\
Hence, we extend  ${\cal I}_{\cal S}$ to ${\mathbb U}$ as in Definition~\ref{defint}.\ref{definttwo}.
\item\label{keyone} If $U \in {\mathbb U}^+$, then ${\cal I}_{\cal S}(U)
\subseteq {\cal M}^{\cal S}$.
\item\label{keytwo} If $U \in {\mathbb U}^-$, then  ${\cal M}^{\cal S} \subseteq 
{\cal I}_{\cal S}(U)$.
\end{enumerate}
\end{lemma}

\begin{proof}
\ref{keyzero}.\ Easy since ${\cal I}_{\cal S}(a) = {\cal M}^{\cal
S}$ which is $\be$-saturated (use theorem~\ref{conf}.\ref{confone}).\\
We show \ref{keyone} and \ref{keytwo} by simultaneous induction on $U$.
\begin{itemize}
\item[\ref{keyone}.] Let $U \in {\mathbb U}^+$ and $M \in {\cal I}_{\cal S}(U)$.
\begin{itemize}
\item If $U = a$, the result comes by definition of $ {\cal I}_{\cal
S}$.
\item If $U = U_1 \sqcap U_2$ and $U_1 \in {\mathbb U}^+$,
then $M \in {\cal I}_{\cal S}(U_1)$ and, by IH, $M \in {\cal
M}^{\cal S}$.
\item If $U = V \f T$, $V \in {\mathbb U}^-$ and $T \in {\mathbb
U}^+$, then let $x \in {\cal S}$. We have $x \in {\cal M}^{\cal
S}$, then, by IH, $x \in {\cal I}_{\cal S}(V)$ and $M x \in {\cal
I}_{\cal S}(T)$. By IH, $M x \in {\cal M}^{\cal S}$, then $M x
\rb^* N$ and $FV(N) \subseteq {\cal S}$. We examine the reduction
$M x \rb^* N$.
\begin{itemize}
\item If $M \rb^* P$ and $N = P x$, then $FV(P) \subseteq FV(N)
 \subseteq {\cal S}$.
 \item If $M \rb^* \l y. Q$ and $Q[y := x] \rb^* N$, then \\$M \rb^* \l y.
 Q = \l x. Q[y := x] \rb^* \l x. N$ and $FV(\l x. N) \subseteq FV(N)
 \subseteq {\cal S}$.
\end{itemize}
Then $M \rb^* M'$ and $FV(M') \subseteq {\cal S}$. Thus $M \in
{\cal M}^{\cal S}$.
\end{itemize}
\item[ \ref{keytwo}.] Let $U \in {\mathbb U}^-$ and $M \in {\cal M}^{\cal S}$.
\begin{itemize}
\item If $U = a$, the result comes by definition of $ {\cal I}_{\cal
S}$.
\item If $U = \o$, then $M \in {\cal I}_{\cal S}(U) = {\cal M}$.
\item If $U = U_1 \sqcap U_2$ and $U_1,U_2 \in {\mathbb U}^-$,
then, by IH, $M \in {\cal I}_{\cal S}(U_1)$ and $M \in {\cal
I}_{\cal S}(U_2)$, then $M \in {\cal I}_{\cal S}(U_1 \sqcap U_2)$.
\item If $U = V \f T$, $V \in {\mathbb U}^+$ and $T \in {\mathbb
U}^-$, then let $P \in {\cal I}_{\cal S}(V)$. We have $M \rb^* N$
and $FV(N) \subseteq {\cal S}$. By IH, $P \in {\cal M}^{\cal S}$,
then $P \rb^* Q$ and $FV(Q) \subseteq {\cal S}$. We have $M P
\rb^* N Q$ and $FV(N Q) = FV(N) \cup FV(Q) \subseteq {\cal S}$,
then $M P \in {\cal M}^{\cal S}$, and, by IH, $M P \in {\cal
I}_{\cal S}(T)$. Thus $M \in {\cal I}_{\cal S}(V \f T)$.
\end{itemize}
\end{itemize}
\end{proof}
The next corollary shows that if $U\in{\mathbb U}^+$ then $[U]$
contains only elements which $\beta$-reduce to closed terms and $[U]$
is the set of all terms that $\be$-reduce to closed terms typable by
$U$.  Note that in the proof of~\ref{comple2one} below, we need
$\beta$-saturation and that this is the reason why we adopted
exclusively $\be$-saturation since remark~\ref{remarkme}.

\begin{corollary}\label{key2}
Let $U \in {\mathbb U}^+$.
\begin{enumerate}
\item\label{key2one}
If $M \in [U]$, then $M \rb^* N$ and $N$
is closed.
\item
\label{comple2one} $[U] = \{M \in {\cal M}$ / $M \rb^* N$ and $N : \< () \v  U\> \}$.
\end{enumerate}
\end{corollary}
\begin{proof}
\begin{itemize}
\item[\ref{key2one}.]
Let ${\cal S}\subseteq {\cal V}$ 
such that ${\cal S} \not = \emptyset$ and ${\cal
S} \cap FV(M) = \emptyset$. Since $M \in [U]$, then $M \in {\cal
I}_{\cal S}(U)$, and, by lemma~\ref{key}, $M \rb^* N$ and $FV(N)
\subseteq {\cal S}$. But, by theorem~\ref{conf}.\ref{confone},
$F(N) \subseteq FV(M)$, then $FV(N) = \emptyset$.
\item[\ref{comple2one}.]
Let $M \in [U]$. By lemma~\ref{comp}.\ref{comptwo}, $M : \<\G \v U\>$.
By~\ref{key2one}, $M \rb^* N$ and $N$ is
closed.  Hence by subject reduction for $\beta$
corollary~\ref{finalbetaeta}, $N : \<\G \r_N \v U\>$. Since $N$ is
closed $N : \<() \v U\>$.

Conversely, let $M$ such that $M \rb^* N$ and $N : \< () \v U\>$, and
take a $\be$-interpretation ${\cal I}$. By theorem \ref{adeq}, $N \in
{\cal I}(U)$ and, since ${\cal I}(U)$ is $\be$-saturated, $M \in {\cal
I}(U)$. Then $M \in \bigcap_{{\cal I} \; \be-interpretation} {\cal
I}(U)$ and so, $M \in [U]$.  \end{itemize} \end{proof}

\begin{remark} Note that neither strong nor weak normalisation holds
in general for typable terms.  For example, $(\lambda x.xx)(\lambda
x.xx) : \<() \v \omega \>$.  As another example, take $\lambda
y.y((\lambda x.xx)(\lambda x.xx)) : \<() \v (\omega \rightarrow a)
\rightarrow a\>$ by lemma~\ref{example}.

We
cannot even establish a strong normalisation result for positive
types.  For example, $(\lambda y.\lambda x.x)((\lambda x.xx)(\lambda
x.xx)) : \<() \v a \rightarrow a\>$. In what follows however, we will
establish a weak normalisation result for positive types.
\end{remark} \begin{definition} We define the function ${\cal I} :
{\cal A} \mapsto {\cal P}({\cal M})$ by: $\forall~a \in {\cal A}$,
${\cal I}(a) = {\cal N}$ where ${\cal N}$ is the set of
$\beta$-normalising terms.  \end{definition}
\begin{lemma}\label{newkey} \begin{enumerate} \item\label{newkeyzero}
${\cal I}$ is a $\be$-interpretation.  I.e., $\forall~a \in {\cal A}$,
${\cal I}(a)$ is $\be$-saturated.\\ Hence, we extend ${\cal I}$ to
${\mathbb U}$ as in Definition~\ref{defint}.\ref{definttwo}.
\item\label{newkeyone} If $U \in {\mathbb U}^+$, then ${\cal I}(U)
\subseteq {\cal N}$.  \item\label{newkeytwo} Let ${\cal
N}'=\{xM_1\dots M_n \in {\cal M} / x\in {\cal V} \mbox{ and } M_1\dots
M_n \in {\cal N} \}$. Note, ${\cal N}' \subseteq {\cal N}$.\\ If $U
\in {\mathbb U}^-$, then ${\cal N}' \subseteq {\cal I}(U)$.
\end{enumerate} \end{lemma} \begin{proof} \ref{newkeyzero} is obvious.
We show \ref{newkeyone} and \ref{newkeytwo} by simultaneous induction
on $U$.  \begin{itemize} \item[\ref{newkeyone}.] Let $U \in {\mathbb
U}^+$ and $M \in {\cal I}(U)$.  \begin{itemize} \item If $U = a$, the
result comes by definition of $ {\cal I}$.  
\item If $U = U_1 \sqcap
U_2$ and $U_1 \in {\mathbb U}^+$, then $M \in {\cal I}(U_1)$ and, by
IH, $M \in {\cal N}$.  
\item If $U = V \f T$, $V \in {\mathbb U}^-$ and $T \in {\mathbb
U}^+$, then let $x \in {\cal V}\subseteq {\cal N}'$ such that $x \not
\in FV(M)$. By IH, $x \in {\cal I}(V)$ and $M x \in {\cal I}(T)$.  By
IH, $M x \in {\cal N}$. Hence, by theorem~\ref{conf}.\ref{conftwo''},
$M \in {\cal N}$.
\end{itemize}
\item[ \ref{newkeytwo}.] Let $U \in {\mathbb U}^-$ and $M \in {\cal
N}'$.  \begin{itemize} \item If $U = a$, the result comes by
definition of $ {\cal I}$.  \item If $U = \o$, then $M \in {\cal I}(U)
= {\cal M}$.  \item If $U = U_1 \sqcap U_2$ and $U_1,U_2 \in {\mathbb
U}^-$, then, by IH, $M \in {\cal I}(U_1)$ and $M \in {\cal I}(U_2)$,
then $M \in {\cal I}(U_1 \sqcap U_2)$.  \item If $U = V \f T$, $V \in
{\mathbb U}^+$ and $T \in {\mathbb U}^-$, then let $P \in {\cal
I}(V)$.  We have $M = xM_1 \dots M_n$ where $M_i \in {\cal N}$ for $1
\leq i \leq n$.  By IH, $P \in {\cal N}$.  Hence, $MP \in {\cal N}'$
and by IH, $M P \in {\cal I}(T)$.  Thus $M \in {\cal I}(V \f T)$.
\end{itemize} \end{itemize} \end{proof}

The next corollary shows that if $U\in{\mathbb U}^+$ then $[U]$
contains only elements which are normalisable.

\begin{corollary}\label{newkey2}
Let $U \in {\mathbb U}^+$.
\begin{enumerate}
\item\label{newkey2one}
If $M \in [U]$, then $M$ is normalisable.
\item\label{newnew}
If $M : \< () \v  U\>$ then $M$ is normalisable.
\item
\label{newcomple2one} $[U] = \{M \in {\cal M}$ / $M \rb^* N$, $N$ is in normal form and $N : \< () \v  U\> \}$.
\end{enumerate}
\end{corollary}
\begin{proof}
\begin{itemize}
\item[\ref{newkey2one}.]
By lemma~\ref{newkey},  $M \in [U] \subseteq {\cal I}(U) \subseteq {\cal N}$.
\item[\ref{newnew}]
By Theorem~\ref{adeq}, $M \in {\cal I}(U)$.  By lemma~\ref{newkey}, $M \in {\cal N}$.
\item[\ref{newcomple2one}.]
Let $M \in [U]$. By Corollary~\ref{key2}.\ref{comple2one},
$M \rb^* P$ and $P : \< () \v  U\>$.
Since by~\ref{newkey2one}, $M$ is normalisable then by Church-Rosser $P$ is normalising. Let $N$ be the normal form of $P$.  By Subject reduction corollary~\ref{finalbetaeta},
$N  : \< () \v  U\>$.\\
The inverse inclusion is obvious by corollary~\ref{key2}.\ref{comple2one}.
\end{itemize}
\end{proof}

\begin{remark}
It should be noted that positive types are not exlusively the types
which satisfy the properties proved about them (e.g.,
corollary~\ref{key2}). For example, let us take the non-positive type
$U' = (\omega \rightarrow b) \rightarrow (a \rightarrow a)$ where $a$
and $b$ are different.  We can show that $[U']$ only contains terms
which reduce to the closed term $\lambda x .\lambda y .y$ (and that
$\lambda x .\lambda y .y:\<()\v U'\>$). Hence, $U'$ is a type which is not positive, yet for which corollary~\ref{key2} holds.  Note that, 
since
$a$ and $b$ are different, then $(\omega \rightarrow b)$ cannot be
used in type derivations.
\end{remark}

\section{Conclusion} 
\label{concl}
In this article, we considered an elegant
intersection type system for which we established basic properties
which include the subject reduction and expansion  properties for $\beta$.
We gave this system a realisability semantics and
we showed its soundness and completeness using a method comparable to
(yet more detailed than) Hindley's completeness semantics for an
earlier intersection type system. The basic difference between both
proofs is that Hindley's notion of saturation is based on equivalence
classes whereas ours is based on a weaker requirement of weak head
normal forms. Hence, all of Hindley's saturated models are also saturated in
our framework yet on the other hand, there are saturated models based
on weak head normal form which cannot be models in Hindley's
framework.  This means that our method provides a larger set of
possible models and this leaves the choice open for better models or
counter-models for particular applications.  We have even proved that
for different notions of saturation (based on weak head reduction and
normal $\beta$-reduction) we obtain the same interpretation for types.
Another difference between our approach and that of Hindley is that
he constructs his models modulo the convertibility relation, whereas
we establish that the interpretation of types is stable by both
$\beta$-reduction and $\beta$-expansion.

Furthermore, we reflected on the meaning of types, especially on the
so-called abstract data types where typability and realisability
coincide.  The presence of $\omega$ in intersection type systems
prevents typability and realisability from coinciding as one sees for
example in $\lambda x.xP$ (where $P$ may contain free variable and may
not be normalisable) whose type is $(\omega \rightarrow a) \rightarrow
a$.  We found a set of types ${\mathbb U}^+$ for which we showed that
typability and realisability coincide.  We have also shown that this
set satisfies the weak normalisation property.

\section*{Acknowledgements} We are grateful for the comments received
from M.\ Dezani, J.R.\ Hindley, V.\  Rahli, J.B.\ Wells and the anonymous referee.

\end{document}